\documentclass[11pt,reqno]{amsart}
\usepackage{amsmath}
\usepackage{url, multirow}

\usepackage{hyperref}

\newcommand{\pFq}[5]{\ensuremath{{}_{#1}F_{#2} \left( \genfrac{}{}{0pt}{}{#3}
{#4} \bigg| {#5} \right)}}
\newcommand{\df}{\dfrac}
\newcommand{\tf}{\tfrac}

 \renewcommand{\a}{\alpha}
\renewcommand{\b}{\beta}
\newcommand{\e}{\epsilon}

\newcommand{\g}{\gamma}
\newcommand{\G}{\Gamma}
\renewcommand{\l}{\lambda}

\renewcommand{\(}{\left\(}
\renewcommand{\)}{\right\)}
\renewcommand{\[}{\left\[}
\renewcommand{\]}{\right\]}
\let\dotlessi=\i
\let\dotlessi=\i

\numberwithin{equation}{section}
 \theoremstyle{plain}
\newtheorem{theorem}{Theorem}[section]

\newtheorem{corollary}[theorem]{Corollary}

\newtheorem{remark}[]{Remark}

   \makeatletter
\def\proof{\@ifnextchar[{\@oproof}{\@nproof}}
\def\@oproof[#1][#2]{\trivlist\item[\hskip\labelsep\textit{#2 Proof of\
#1.}~]\ignorespaces}
\def\@nproof{\trivlist\item[\hskip\labelsep\textit{Proof.}~]\ignorespaces}

\makeatother

\begin{document}
\title[Analogue of a Fock-type integral arising from electromagnetism]{Analogue of a Fock-type integral arising from electromagnetism and its applications in number theory}
\author{Atul Dixit and Arindam Roy}\thanks{2010 \textit{Mathematics Subject Classification.} Primary 11M06, 33E20; Secondary 33C10.\\
\textit{Keywords and phrases.} Bessel functions, generalized sum-of-divisors function, Vorono\"{\dotlessi} summation formula, analytic continuation}
\address{Discipline of Mathematics, Indian Institute of Technology Gandhinagar, Palaj, Gandhinagar 382355, Gujarat, India}\email{adixit@iitgn.ac.in}
\address{Department of Mathematics, University of North Carolina at Charlotte, 9201 University City Blvd., Charlotte, NC 28223, USA} \email{aroy15@uncc.edu}
\begin{abstract}
Closed-form evaluations of certain integrals of $J_{0}(\xi)$, the Bessel function of the first kind, have been crucial in the studies on the electromagnetic field of alternating current in a circuit with two groundings, as can be seen from the works of Fock and Bursian, Schermann etc. Koshliakov's generalization of one such integral, which contains $J_s(\xi)$ in the integrand, encompasses several important integrals in the literature including Sonine's integral. Here we derive an analogous integral identity where $J_{s}(\xi)$ is replaced by a kernel consisting of a combination of $J_{s}(\xi)$, $K_{s}(\xi)$ and $Y_{s}(\xi)$ that is of utmost importance in number theory. Using this identity and the Vorono\"{\dotlessi} summation formula, we derive a general transformation relating infinite series of products of Bessel functions $I_{\l}(\xi)$ and $K_{\l}(\xi)$ with those involving the Gaussian hypergeometric function. As applications of this transformation, several important results are derived, including what we believe to be a corrected version of the first identity found on page $336$ of Ramanujan's Lost Notebook.
\end{abstract}
\maketitle

\section{Introduction}\label{intro}
In his famous memoir on the propagation of waves in wireless telegraphy,  Sommerfeld \cite{sommerfeld} developed a method of integral representation for calculating the electromagnetic field on a flat boundary where the solution sought is expressed in terms of a Fourier integral consisting of Bessel functions. The Sommerfeld integral, given by \cite[p.~366]{praus}
\begin{equation}\label{popovintfock}
\int_{0}^{\infty}tJ_{0}(\rho t)\frac{e^{-a\sqrt{t^2+\xi^2}}}{\sqrt{t^2+\xi^2}}\, dt=\frac{e^{-\xi\sqrt{a^2+\rho^2}}}{\sqrt{a^2+\rho^2}},
\end{equation}
is valid for Re$(a)>|$Im$(\rho)|$ and Re$(\xi)>0$, and is actually a special case of \cite[p.~416, Equation (2)]{watson-1966a} (see also \cite[p.~693, formula 6.596.7]{grn})
\begin{equation}\label{grn693}
\int_{0}^{\infty}t^{s+1}J_{s}(\rho t)\frac{K_{\nu}\left(a\sqrt{t^2+\xi^2}\right)}{(t^2+\xi^2)^{\nu/2}}\, dt=\frac{\rho^{s}}{a^{\nu}}\left(\frac{\sqrt{a^2+\rho^2}}{\xi}\right)^{\nu-s-1}K_{\nu-s-1}\left(\xi\sqrt{a^2+\rho^2}\right),
\end{equation}
where\footnote{These conditions, which are more general than the ones given in \cite[p.~416, Equation (2)]{watson-1966a} and \cite[p.~693, formula 6.596.7]{grn}, are established in Section \ref{ex}.} Re$(a)>|$\textup{Im}$(\rho)|$, $|\arg\rho|<\pi$, Re$(s)>-1$, and Re$(\xi)>0$. Here $J_{s}(\xi)$, the Bessel function of the first kind of order $s$, is defined by \cite[p.~40]{watson-1966a}
\begin{align}\label{sumbesselj}
	J_{s}(\xi):=\sum_{m=0}^{\infty}\frac{(-1)^m(\xi/2)^{2m+s}}{m!\Gamma(m+1+s)}, \quad |\xi|<\infty,
	\end{align}
and $K_{s}(\xi)$ is the modified Bessel function of the second kind of order $s$ defined by \cite[p.~78, eq.~(6)]{watson-1966a},
\begin{equation*}
K_{s}(\xi):=\frac{\pi}{2}\frac{\left(I_{-s}(\xi)-I_{s}(\xi)\right)}{\sin\pi s},
\end{equation*}
where $I_{s}(\xi)$ is the modified Bessel function of the first kind of order $s$ given by \cite[p.~77]{watson-1966a}
\begin{equation}\label{besseli}
I_{s}(\xi):=
\begin{cases}
e^{-\frac{1}{2}\pi s i}J_{s}(e^{\frac{1}{2}\pi i}\xi), & \text{if $-\pi<$ arg $\xi\leq\frac{\pi}{2}$,}\\
e^{\frac{3}{2}\pi s i}J_{s}(e^{-\frac{3}{2}\pi i}\xi), & \text{if $\frac{\pi}{2}<$ arg $\xi\leq \pi$}.
\end{cases}
\end{equation}
Seventeen years after Sommerfeld's aforementioned work, Fock and Bursian \cite[p.~361--363]{fockbursian} encountered integrals of the type of \eqref{popovintfock} in their study on the electromagnetic field of alternating current in a circuit with two groundings. Their study rested on the following important result \cite[Equation (20)]{fockbursian} valid for Re$(z)>$ Re$(w)\geq 0$ and Re$(\nu)>-3/4$:
\begin{align}\label{fock}
\int_{0}^{\infty}tJ_{0}(\rho t)\left(\frac{\sqrt{t^2+z^2}-\sqrt{t^2+w^2}}{\sqrt{t^2+z^2}+\sqrt{t^2+w^2}}\right)^{\nu}\frac{\, dt}{\sqrt{t^2+z^2}\sqrt{t^2+w^2}}=I_{\nu}\left(\frac{\rho(z-w)}{2}\right)K_{\nu}\left(\frac{\rho(z+w)}{2}\right).
\end{align}
Identity \eqref{fock} was given without proof in \cite{fockbursian} and it was claimed there that its proof will be published later. Seven years later, Fock proved it in his paper \cite{fock}. The integrals of Fock have been successfully used to explore the electrical properties of earth's subsurface, for example, see \cite{praus}. See also Schermann's work on electromagnetism \cite[Equation (11)]{schermann} which makes use of \eqref{fock}.

Popov \cite{popovfock} obtained a short proof of \eqref{fock} and Koshliakov \cite{kosh34} obtained its remarkable generalization, namely, for Re$(s)>-1$, Re$\left(s+2\nu+\tfrac{3}{2}\right)>0$ and Re$(z)>$ Re$(w)\geq0$,
\begin{align}\label{koshfock}
&\int_{0}^{\infty}t^{s+1}J_{s}(\rho t)\left(\frac{\sqrt{t^2+z^2}-\sqrt{t^2+w^2}}{\sqrt{t^2+z^2}+\sqrt{t^2+w^2}}\right)^{\nu}\frac{1}{\sqrt{t^2+z^2}
\sqrt{t^2+w^2}}\left(\frac{1}{\sqrt{t^2+z^2}}+\frac{1}{\sqrt{t^2+w^2}}\right)^{2s}\nonumber\\
&\quad\times\pFq21{\nu-s, -s}{\nu+1}{\left(\frac{\sqrt{t^2+z^2}-\sqrt{t^2+w^2}}{\sqrt{t^2+z^2}+\sqrt{t^2+w^2}}\right)^{2}}\, dt\nonumber\\
&=\frac{\G(\nu+1)}{\G(\nu+s+1)}(2\rho)^{s}I_{\nu}\left(\frac{\rho(z-w)}{2}\right)K_{\nu}\left(\frac{\rho(z+w)}{2}\right),
\end{align}
where the hypergeometric function $_2F_1$ is defined by
\begin{equation*}
_2F_1\left(\begin{matrix} a,b\\c\end{matrix}\,\Bigr|\, \xi\right):=\sum_{n=0}^{\infty}\df{(a)_n(b)_n}{(c)_nn!}\xi^n,\qquad |\xi|<1,
\end{equation*}
with $(a)_n:=a(a+1)\cdots(a+n-1)$.

Indeed, if we let $s=0$ in \eqref{koshfock}, we recover \eqref{fock}. As Koshliakov remarks \cite[p.~146]{kosh34}, the above integral contains many important integral evaluations apart from \eqref{fock}. For example, if we let $s=\nu$, divide both sides of \eqref{koshfock} by $(z-w)^{\nu}$ and then let $z\to w$, he records that one obtains Sonine's integral, namely, for Re$(\rho)>0$, Re$(w)>0$ and Re$(\nu)>-1/2$,
\begin{equation}\label{sonine}
\int_{0}^{\infty}\frac{x^{\nu+1}J_{\nu}(\rho x)}{(x^2+w^2)^{2\nu+1}}\, dx=\left(\frac{\rho}{2\sqrt{w}}\right)^{2\nu}\frac{K_{\nu}(\rho w)}{\G(2\nu+1)}.
\end{equation}

In a recent paper \cite[Theorem 1.6]{bdkz}, Berndt, Kim, Zaharescu and one of the present authors successfully employed \eqref{koshfock} in the Vorono\"{\dotlessi} summation formula for $r_k(n)$, the number of representations of $n$ as sum of $k$ squares, where $k\geq2$, to obtain an important transformation which gives, as corollaries, several important results in analytic number theory, for example, those of Dixon and Ferrar \cite[Equation (3.12)]{dixfer2}, Hardy \cite[Equation (2.12)]{hardyqjpam1915} and Popov \cite[Equation (6)]{popov1935}.

Loosely speaking, the Vorono\"{\dotlessi} summation formula provides an avenue to investigate the order of magnitude of the error term in the asymptotic estimates of the summatory function of the product of an arithmetic function $a(n)$ and a function $f$ by linking it with certain infinite series whose summand is an integral transform of $f$ with respect to a certain kernel. For different arithmetic functions $a(n)$, the function $f$ needs to satisfy corresponding hypotheses which validate these summation formulas.

The Vorono\"{\dotlessi} summation formula for $a(n)=r_2(n)$ is given in \cite[p.~274]{landau} (or \cite[Thm.~A]{dixfer2}), whereas for $k\geq 2$, it is given by Popov \cite[Equation (3)]{popov}, and in an equivalent form by Guinand \cite[p.~264]{guinandconcord}. This equivalent form is
\begin{align}\label{guirknsuminfallk}
&\sum_{n=1}^{\infty}r_k(n)n^{\frac{1}{2}-\frac{k}{4}}f(n)-\frac{\pi^{\frac{k}{2}}}{\G(\frac{k}{2})}\int_{0}^{\infty}
x^{\frac{k}{4}-\frac{1}{2}}f(x)\, dx\nonumber\\
&=\sum_{n=1}^{\infty}r_k(n)n^{\frac{1}{2}-\frac{k}{4}}g(n)-\frac{\pi^{\frac{k}{2}}}{\G(\frac{k}{2})}\int_{0}^{\infty}
x^{\frac{k}{4}-\frac{1}{2}}g(x)\, dx,
\end{align}
where
\begin{equation*}
g(y)=\pi\int_{0}^{\infty}f(t)J_{\frac{k}{2}-1}(2\pi\sqrt{yt})\, dt.
\end{equation*}
Unfortunately, neither Popov nor Guinand gives any conditions on $f$ and $g$. One can refer to \cite[Theorem 1.5]{bdkz} for the same and also for a detailed discussion on this topic. 

While the integral transform of $f$ in \eqref{guirknsuminfallk} involves the Bessel function $J_{s}(t)$ in the kernel, the corresponding summation formula for the generalized sum-of-divisors function $\sigma_{-s}(n):=\sum_{d|n}d^{-s}$ involves the Koshliakov kernel 
\begin{equation}\label{kk}
F_{s}(t):=M_{s}(t)\cos\left(\frac{\pi s}{2}\right)-J_{s}(t)\sin\left(\frac{\pi s}{2}\right),
\end{equation}
where
\begin{equation}\label{msxi}
M_{s}(\xi):=\frac{2}{\pi}K_{s}(\xi)-Y_{s}(\xi),
\end{equation}
with $Y_{s}(\xi)$ being the Bessel function of the second kind defined by \cite[p.~64]{watson-1966a}
\begin{align*}
Y_{s}(\xi)=\frac{J_{s}(\xi)\cos(\pi s)-J_{-s}(\xi)}{\sin{\pi s}}.
\end{align*}
For a comprehensive discussion and recent new results on the Vorono\"{\dotlessi} summation formula for $\sigma_{-s}(n)$, the reader is referred to a paper by Berndt, Zaharescu and the present authors \cite[Section 6]{bdrz}. 

Guinand's version \cite[Theorem 6]{guinandsumself2}, \cite[Equation (1)]{guinand} of the Vorono\"{\dotlessi} summation formula for $\sigma_{-s}(n)$ reads, if $f(x)$ and  $f'(x)$ are integrals, $f$ tends to zero as $x\to\infty$, $f(x), xf'(x)$, and $x^2f''(x)$ belong to $L^{2}(0,\infty)$, $-\frac{1}{2}<$ Re$(s)<\frac{1}{2}$, and
\begin{equation}\label{recpairgui}
g(x) = 2\pi \int_{0}^{\infty} f(t) F_s(4\pi\sqrt{xt}) \,dt,
\end{equation}
where $F_s(t)$ is defined in \eqref{kk}, then
\begin{align}\label{guitra}
&\sum_{n=1}^{\infty} \sigma_{-s} (n) n^{\frac{s}{2}} f(n) - \zeta(1+s) \int_{0}^{\infty} y^{\frac{s}{2}} f(y) \,dy - \zeta(1-s) \int_{0}^{\infty} y^{-\frac{s}{2}} f(y) \,dy\nonumber\\
&= \sum_{n=1}^{\infty} \sigma_{-s} (n) n^{\frac{s}{2}} g(n) - \zeta(1+s) \int_{0}^{\infty} y^{\frac{s}{2}} g(y) \,dy - \zeta(1-s) \int_{0}^{\infty} y^{-\frac{s}{2}} g(y) \,dy.
\end{align}
The Vorono\"{\dotlessi} summation formula for $\sigma_0(n)=d(n)$, the number of divisors of $n$, has been found to be enormously useful in physics; for example, S.~Egger and F.~Steiner \cite{es1, es2} have shown that it plays the role of an exact trace formula for a Schr\"{o}dinger operator on a certain non-compact quantum graph.

The two notoriously difficult, and still unsolved problems, in analytic number theory are the Gauss circle and the Dirichlet divisor problems. They ask for the correct order of magnitude of the error terms $P(x)$ and $\Delta(x)$ involved in the summatory functions of $r_2(n)$ and $d(n)$ respectively. A natural tool to tackle these problems is to use the Vorono\"{\dotlessi} summation formula corresponding to each of these arithmetic functions. While they are inherently distinct in that one counts the number of lattice points inside a circle and the other the number of lattice points under a hyperbola, they have lots of similarities. For example, if an upper bound on $P(x)$ is obtained, then a similar one can be found for $\Delta(x)$. The reader is referred to a recent survey article by Berndt, Kim and Zaharescu \cite{bkz-unsolved} on these two problems. For the generalized Gauss circle problem, which counts the number of lattice points on multidimensional spheres and where the associated arithmetic function is naturally $r_k(n)$, the reader is referred to the two papers of Lursmana$\check{s}$vili \cite{lurs1}, \cite{lurs2}. 	

Considering the fact that Fock's integral evaluation \eqref{fock} is one of the key steps in his and Bursian's work \cite{fockbursian} on electromagnetism, the first goal of this paper is to derive an analogue of Koshliakov's formula \eqref{koshfock} (of which \eqref{fock} is a special case) where the kernel $J_{s}(\rho t)$ is replaced by the Koshliakov kernel $F_s(\rho t)$. The second goal is to apply this evaluation in \eqref{guitra} to obtain a general summation formula from which many interesting and important identities result as special cases.  

However, before we proceed to establish the former, it is important to mention the thoughts of Dixon and Ferrar \cite[p.~168]{dixfer3} on obtaining analogues of known integral identities where the $J$-Bessel function is replaced by the corresponding kernel $M_0(\xi)$, where $M_s(\xi)$ is defined in \eqref{msxi}. They say\footnote{We have inserted some comments in square brackets and changed some notations to conform them with those used in our paper.}:\\

\textit{Our recent work on summation-formulae suggests that it would be of interest to find a formula corresponding to (1) [Equation \eqref{grn693} here] when $J_{s}(\rho t)t^{s+1}$ is replaced by $tM_{0}(\rho t)$. The problem of finding such a formula has proved to be a difficult one: we have found that any single integral of the type which occurs in (1) leads to a complicated and ill-balanced formula; the only way we have discovered of obtaining a reasonably well-balanced formula is to combine two integrals.}

To give an example of their principle, consider the following identity due to Koshliakov \cite[Equation (5)]{koshliakov} for Re$(\b)>0$:
\begin{equation}\label{koshvor}
2 \sum_{n=1}^{\infty} d(n) \left( K_{0} \left( 4 \pi e^{\frac{i \pi}{4}} \sqrt{n\b} \right) + K_{0} \left( 4 \pi e^{-\frac{i \pi}{4}} \sqrt{n\b} \right)\right)=- \gamma - \frac{1}{2} \log\b - \frac{1}{4 \pi\b} +
\frac{\b}{\pi} \sum_{n=1}^{\infty} \frac{d(n)}{\b^{2} + n^{2}},
\end{equation}
which is instrumental in his clever proof of the Vorono\"{\dotlessi} summation formula for the divisor function $d(n)$. The latter is, in fact, shown in \cite{soni} to be equivalent to the above formula. Here $\g$ denotes Euler's constant. Koshliakov's identity can be proved by replacing $a$, first by $\b e^{i\frac{\pi}{2}}$, and then by $\b e^{-i\frac{\pi}{2}}$, in the identity \cite[Equations (5), (6)]{voronoi}, \cite[p.~254]{lnb}
\begin{equation}\label{vorram}
2\sum_{n=1}^{\infty}d(n)K_{0}(4\pi\sqrt{an})=\frac{a}{\pi^2}\sum_{n=1}^{\infty}\frac{d(n)\log(a/n)}{a^2-n^2}-\frac{\gamma}{2}-\left(\frac{1}{4}+\frac{1}{4\pi^2a}\right)\log a-\frac{\log 2\pi}{2\pi^2a},
\end{equation}
and then by adding the two resulting identities. Note that even though each of the individual series $\sum_{n=1}^{\infty} d(n) K_{0} \left( 4 \pi e^{\pm\frac{i \pi}{4}} \sqrt{n\b} \right)$ admits a representation through \eqref{vorram}, adding the two gives a simpler and a well-balanced formula \eqref{koshvor}.

A generalization of Koshliakov's above identity, which is also a key ingredient in obtaining a short proof of Vorono\"{\dotlessi} summation formula for $\sigma_{\l}(n)$ \cite[Section 6.1]{bdrz}, was first studied by Moll and one of the present authors in \cite[Section 6]{dixitmoll}. For Re$(\b)>0$ and Re$(\l)>-1$, it is given by
\begin{align}\label{dixmol}
&2 \sum_{n=1}^{\infty} \sigma_{-\l}(n) n^{\frac{\l}{2}} 
\left( e^{\frac{\pi i \l}{4}} K_{\l}( 4 \pi e^{\frac{\pi i}{4}} \sqrt{n\b} ) +
 e^{-\frac{\pi i \l}{4}} K_{\l}( 4 \pi e^{-\frac{\pi i}{4}} \sqrt{n\b} ) \right)\nonumber\\
&=- \frac{\Gamma(\l) \zeta(\l)}{(2 \pi \sqrt{\b})^{\l}}  +
\frac{\b^{\frac{\l}{2}-1}}{2 \pi} \zeta(\l) - \frac{\b^{\frac{\l}{2}}}{2} \zeta(\l+1)  +\frac{\b^{\frac{\l}{2}+1}}{\pi} \sum_{n=1}^{\infty} \frac{\sigma_{-\l}(n)}{n^{2}+\b^{2}},
\end{align}
Here one can observe the same aforementioned principle. To see this, we first note the following generalization of \eqref{vorram}, due to Cohen \cite[Theorem 3.4]{cohen}, valid for Re$(a)>0$ and $\l\notin\mathbb{Z}$:
\begin{align*}
&8\pi a^{\l/2}\sum_{n=1}^{\infty}\sigma_{-\l}(n)n^{\l/2}K_{\l}(4\pi\sqrt{n a})=
A(\l, a)\zeta(\l)+B(\l, a)\zeta(\l+1)\\
&\quad+\frac{2}{\sin\left(\pi \l/2\right)}\left(\sum_{1\leq j\leq k}\zeta(2j)\zeta(2j-\l)a^{2j-1}+a^{2k+1}\sum_{n=1}^{\infty}\sigma_{-\l}(n)\frac{n^{\l-2k}-a^{\l-2k}}{n^2-a^2}\right),\notag
\end{align*}%
where
\begin{align*}
A(\l, a)&=\frac{a^{\l-1}}{\sin\left(\pi \l/2\right)}-(2\pi)^{1-\l}\Gamma(\l),\nonumber\\
B(\l, a)&=\frac{2}{a}(2\pi)^{-\l-1}\Gamma(\l+1)-\frac{\pi a^{\l}}{\cos\left(\pi \l/2\right)}.
\end{align*}
Now as shown in \cite[p.~844]{bdrz}, \eqref{dixmol} follows from the above result by first replacing $a$ by $i\b$ for $-\pi<\arg\b<\frac{1}{2}\pi$, then by $-i\b$ for $-\frac{1}{2}\pi<\arg\b<\pi$, and then by adding the resulting two identities. Now observe that compared to the representations for each of $\sum_{n=1}^{\infty} \sigma_{-\l}(n) n^{\frac{\l}{2}} 
K_{\l}( 4 \pi e^{\pm\frac{\pi i}{4}} \sqrt{n\b} )$, the one derived in \eqref{dixmol} for its left-hand side is much simpler and well-balanced.

Dixon and Ferrar's principle was also known to Ramanujan, for, on page $336$ of his Lost Notebook \cite{lnb}, we find the following identity: 
{\allowdisplaybreaks\begin{align}\label{ramin}
&\Gamma\left(s+\frac{1}{2}\right)
\bigg\{\frac{\zeta(1-s)}{(s-\frac{1}{2})\b^{s-\frac{1}{2}}}+\frac{\zeta(-s)\tan\frac{1}{2}\pi s}{2\b^{s+\frac{1}{2}}}\nonumber\\
&\quad+\sum_{n=1}^{\infty}\frac{\sigma_s(n)}{2i}\left((\b-in)^{-s-\frac{1}{2}}-(\b+in)^{-s-\frac{1}{2}}\right)\bigg\}
\nonumber\\
&=(2\pi)^s\bigg\{\frac{\zeta(1-s)}{2\sqrt{\pi\b}}-2\pi\sqrt{\pi\b}\zeta(-s)\tan\tfrac{1}{2}\pi s\nonumber\\
&\quad+\sqrt{\pi}\sum_{n=1}^{\infty}\frac{\sigma_s(n)}{\sqrt{n}}e^{-2\pi\sqrt{2n\b}}
\sin\left(\frac{\pi}{4}+2\pi\sqrt{2n\b}\right)\bigg\}.
\end{align}}
Clearly, one can see powers of $\b+in$ and $\b-in$ both playing a role in the above identity than just one of them. Note that in the summands of the infinite series on the right-hand sides of each of \eqref{koshvor} and \eqref{dixmol}, one has
\begin{equation*}
\frac{1}{\b^2+n^2}=\frac{1}{2\b}\left(\frac{1}{\b+in}+\frac{1}{\b-in}\right).
\end{equation*}
The identity \eqref{ramin}, however, incorrect since the series on its left-hand side diverges. We obtain a corrected version of \eqref{ramin} in Corollary \ref{ramanujan-type} of this paper. In fact, searching for a corrected version of \eqref{ramin} was one of the chief motivations behind our work.

Also, among other things, we carry out an in-depth analysis of the series
\begin{align*}
\sum_{n=1}^{\infty}\sigma_s(n)\bigg\{&I_{\l}\left(2\pi e^{-\frac{i\pi}{4}}\sqrt{n}(\sqrt{\a}-\sqrt{\b})\right)K_{\l}\left(2\pi e^{-\frac{i\pi}{4}}\sqrt{n} (\sqrt{\a}+\sqrt{\b})\right)\nonumber\\
&+I_{\l}\left(2\pi e^{\frac{i\pi}{4}}\sqrt{n}(\sqrt{\a}-\sqrt{\b})\right)K_{\l}\left(2\pi e^{\frac{i\pi}{4}}\sqrt{n}(\sqrt{\a}+\sqrt{\b})\right)\bigg\},
\end{align*}
for $|\arg\a|<\pi/2$, $|\arg\b|<\pi/2$ and Re$(\sqrt{\a})> $Re$(\sqrt{\b})$, whose limiting case $\a\to\b$ (after we divide it by $(\sqrt{\a}-\sqrt{\b})^{\l}$) is, as will be shown in the paper, the series on the left-hand side of \eqref{dixmol}. The convergence of the above series for the aforementioned values of $\a, \b$ can be seen from the fact that for large $n$
\begin{align*}
&I_{\l}\left(2\pi\overline{\epsilon}\sqrt{n}(\sqrt{\a}-\sqrt{\b})\right)K_{\l}\left(2\pi\overline{\epsilon}\sqrt{n} (\sqrt{\a}+\sqrt{\b})\right)\nonumber\\
&+I_{\l}\left(2\pi\epsilon\sqrt{n}(\sqrt{\a}-\sqrt{\b})\right)K_{\l}\left(2\pi\epsilon\sqrt{n}(\sqrt{\a}+\sqrt{\b})\right)\nonumber\\
&\sim\frac{e^{-2\pi\sqrt{2n\b}}}{2\pi\sqrt{n}\sqrt{\a-\b}}\sin\left(\frac{\pi}{4}-2\pi\sqrt{2n\b}\right),
\end{align*}
which, in turn, can be proved using \eqref{asymbess2} and \eqref{iasy} below.

Dixon and Ferrar \cite[p.~174]{dixfer3} gave the following integral evaluation which corresponds to \eqref{grn693} upon replacement of the kernel $J_{s}(\rho t)t^{s+1}$ by $tM_{0}(\rho t)$:
\begin{align}\label{st1}
&\int_{0}^{\infty}tM_{0}(\rho t)\left\{\frac{K_{\nu}\left(a\sqrt{\xi^2+t^2}\right)}{(\xi^2+t^2)^{\nu/2}}+\frac{K_{\nu}\left(a\sqrt{\xi^2-t^2}\right)}{(\xi^2-t^2)^{\nu/2}}\right\}\, dt\nonumber\\
&=\frac{i\xi^{1-\nu}}{a^{\nu}}\left\{\frac{K_{1-\nu}\left(\xi\sqrt{a^2+\rho^2}\right)}{(a^2+\rho^2)^{(1-\nu)/2}}+\frac{K_{1-\nu}\left(\xi\sqrt{a^2-\rho^2}\right)}{(a^2-\rho^2)^{(1-\nu)/2}}\right\}.
\end{align}
where $\xi>0, a>0, \rho>0, a\neq \rho$, and Re $\nu>0$. (They also give different extensions of the above formula for values of $\nu, a$ and $\rho$ other than those considered above, which ensure its validity as it is, or with some modifications.) Here, of course, the path of integration is indented to avoid the singularity of the integrand. It ensures that $\arg(z^2-t^2)=-\pi$ when $t>z$. At the end of their paper, they also say that the indented path may be replaced by a straight-line path provided Re$(\nu)<1$ where they always retain the conventions $\arg(z^2-t^2)=0$ when $t<z$ and $\arg(z^2-t^2)=-\pi$ when $t>z$.

Though Dixon and Ferrar evaluated certain integrals with $tM_{0}(\rho t)$ as the kernel, they did not work with integrals having the general kernel $t^{s+1}M_{s}(\rho t)$ in their integrands. However, Koshliakov \cite[p.~425]{kosh1938} obtained such a result\footnote{There are two typos in Koshliakov's version of \eqref{koshcertain}, namely, in \emph{his} version, the arguments of the $K$-Bessel functions inside the integral on the left contain $a$, and those on the right contain $x$. Both these typos are corrected, albeit with renaming of his parameters, in our version of his formula given in \eqref{koshcertain}.}, namely, for $\rho>0, b>0, \xi>0$ and $s>-1$,
\begin{align}\label{koshcertain}
&\int_{0}^{\infty}t^{s+1}M_{s}(\rho t)\left\{\frac{K_{\nu}\left(b \sqrt{\xi^2+it^2}\right)}{(\xi^2+it^2)^{\nu/2}}+\frac{K_{\nu}\left(b \sqrt{\xi^2-it^2}\right)}{(\xi^2-it^2)^{\nu/2}}\right\}\, dt\nonumber\\
&=\frac{\xi^{1+s-\nu}\rho^{s}}{b^{\nu}}\left\{e^{-\frac{i\pi}{4}\left(\nu+s-1\right)}\frac{K_{1+s-\nu}(\xi e^{-\frac{i\pi}{4}} \sqrt{\rho^2+ib^2})}{(\rho^2+ib^2)^{\frac{1+s-\nu}{2}}}+e^{\frac{i\pi}{4}\left(\nu+s-1\right)}\frac{K_{1+s-\nu}(\xi e^{\frac{i\pi}{4}} \sqrt{\rho^2-ib^2})}{(\rho^2-ib^2)^{\frac{1+s-\nu}{2}}}\right\}.
\end{align}
Henceforth, throughout the paper, we follow Koshliakov \cite{kosh1938} in using the notation\footnote{Note that $\epsilon^{-\left(\nu-1\right)}=\overline{\epsilon}^{\left(\nu-1\right)}$, so there is no possibility of confusion unless, of course, one does not put the parentheses. However, we wish to keep the results in the sequel symmetric with respect to $\epsilon$ and $\overline{\epsilon}$ whenever we can. We will use, without mention, the fact that $\overline{\epsilon}=1/\epsilon$.}
\begin{equation*}
\epsilon=e^{\frac{i\pi}{4}},\hspace{5mm} \overline{\epsilon}=e^{-\frac{i\pi}{4}},
\end{equation*}
for brevity and convenience. Note that in \eqref{koshcertain}, Koshliakov has avoided the necessity of indenting the contour, which was the requirement in Dixon and Ferrar's \eqref{st1}. Thus the path of integration in \eqref{koshcertain} is the straight line from $0$ to $\infty$.

\section{Main Results}\label{mr}
All of our results in this paper abide by the principle by Dixon and Ferrar of combining two integrals while working with the kernel $t^{s+1}F_{s}(t)$, which is, of course, a generalization of the kernel $tM_{0}(t)$ that Dixon and Ferrar work with in \cite{dixfer3}. 

An application of \eqref{grn693} and \eqref{koshcertain} leads to the following desired analogue of \eqref{koshfock}.
\begin{theorem}\label{tfockathm}
Let $|\arg z|<\pi/4, |\arg w|<\pi/4$ and \textup{Re}$(z)>$ \textup{Re}$(w)$. Let $F_{s}(t)$ be defined in \eqref{kk}. Define
\begin{align}\label{funa}
\mathcal{A}(t):=\mathcal{A}(s, \l, z, w, t)&:=\frac{1}{\sqrt{z^2+it}\sqrt{w^2+it}}\left(\frac{1}{\sqrt{z^2+it}}+\frac{1}{\sqrt{w^2+it}}\right)^{2s}\left(\frac{\sqrt{z^2+it}-\sqrt{w^2+it}}{\sqrt{z^2+it}+\sqrt{w^2+it}}\right)^{\l}\nonumber\\
&\quad\times\pFq21{\l-s, -s}{\l+1}{\left(\frac{\sqrt{z^2+it}-\sqrt{w^2+it}}{\sqrt{z^2+it}+\sqrt{w^2+it}}\right)^2}.
\end{align}
Then for $\rho>0$, \textup{Re}$(s)>-1$, \textup{Re}$(\l)>-1/2$ and \textup{Re}$(s+\l)>-1$, we have
\begin{align}\label{tfockaeqn}
&\int_{0}^{\infty}t^{s+1}F_{s}(\rho t)\left(\mathcal{A}(t^2)+\mathcal{A}(-t^2)\right)\, dt\nonumber\\
&=\frac{\G(\l+1)}{\G(s+\l+1)}(2\rho)^{s}\left\{I_{\l}\left(\frac{\rho\overline{\epsilon}(z-w)}{2}\right)K_{\l}\left(\frac{\rho\overline{\epsilon}(z+w)}{2}\right)+I_{\l}\left(\frac{\rho\epsilon(z-w)}{2}\right)K_{\l}\left(\frac{\rho\epsilon(z+w)}{2}\right)\right\}.
\end{align}
\end{theorem}
Similar to \eqref{koshfock}, the above theorem also contains many important integral evaluations as its special cases. For example, letting $s=\l$, dividing both sides by $(z-w)^{\l}$ and then letting $z\to w$ leads to the following analogue of Sonine's integral \eqref{sonine} for $\rho>0$, $|\arg w|<\pi/4$ and Re$(\l)>-1/2$:
\begin{align*}
&\int_{0}^{\infty}t^{\l+1}F_{\l}(\rho t)\left\{\frac{1}{(w^2+it)^{2\l+1}}+\frac{1}{(w^2-it)^{2\l+1}}\right\}\, dt\nonumber\\
&=\frac{1}{\G(2\l+1)}\left(\frac{\rho}{2\sqrt{w}}\right)^{2\l}\left\{e^{\frac{i\pi\l}{4}}K_{\l}\left(e^{\frac{i\pi}{4}}\rho w\right)+e^{-\frac{i\pi\l}{4}}K_{\l}\left(e^{-\frac{i\pi}{4}}\rho w\right)\right\}.
\end{align*}
The transformation between two infinite series involving $\sigma_{s}(n)$ that we obtain upon using Theorem \ref{tfockathm} in \eqref{guitra} is now given. The corresponding result for $r_k(n)$, the number of representations of $n$ as sum of $k$ squares, $k\in\mathbb{N}$, was established in \cite[Theorem 1.6]{bdkz}, although we emphasize that the level of difficulty in obtaining the result for $\sigma_s(n)$ is quite high as compared to the one for $r_k(n)$.
\begin{theorem}\label{anathmbdkz}
Let $\mathcal{A}(s, \l, z, w, t)$ be defined in \eqref{funa}. Let $|\arg \a|<\pi/2, |\arg \b|<\pi/2$ with \textup{Re}$(\sqrt{\a})>$ \textup{Re}$(\sqrt{\b})$. Let \textup{Re}$(\l)>0$. Define $h(s,\l)$ by
\begin{equation}\label{hsl}
h(s, \l):=\begin{cases}
\hspace{6mm}0,\hspace{9mm}\text{if}\hspace{2mm}\textup{Re}(s+\l)>0,\\
\displaystyle\frac{(\a-\b)^{\l}}{2^{4\l+1}},\hspace{2mm}\text{if}\hspace{1mm}s=-\l.
\end{cases}
\end{equation}
Then for all $s\in\mathbb{C}$ such that \textup{Re}$(s+\l)>0$ or $s=-\l$, we have
\begin{align}\label{anathmbdkzeqn}
&\sum_{n=1}^{\infty}\sigma_s(n)\bigg\{I_{\l}\left(2\pi\overline{\epsilon}(\sqrt{n\a}-\sqrt{n\b})\right)K_{\l}\left(2\pi\overline{\epsilon}(\sqrt{n\a}+\sqrt{n\b})\right)\nonumber\\
&\qquad\qquad+I_{\l}\left(2\pi\epsilon(\sqrt{n\a}-\sqrt{n\b})\right)K_{\l}\left(2\pi\epsilon(\sqrt{n\a}+\sqrt{n\b})\right)\bigg\}\nonumber\\
&=-\frac{\G(s+\l+1)\zeta(1-s)}{2(8\pi)^s\G(\l+1)}h(s, \l)-\frac{\zeta(-s)}{2\l}\left(\frac{\sqrt{\a}-\sqrt{\b}}{\sqrt{\a}+\sqrt{\b}}\right)^{\l}+\frac{\zeta(-s)}{2^{3s+2}\pi^{s+1}\sqrt{\a\b}}\frac{\G(\l+s+1)}{\G(\l+1)}\nonumber\\
&\quad\times\left(\frac{\sqrt{\a}-\sqrt{\b}}{\sqrt{\a}+\sqrt{\b}}\right)^{\l}\left(\frac{1}{\sqrt{\a}}+\frac{1}{\sqrt{\b}}\right)^{2s}\pFq21{\l-s, -s}{\l+1}{\left(\frac{\sqrt{\a}-\sqrt{\b}}{\sqrt{\a}+\sqrt{\b}}\right)^2}\nonumber\\
&\quad+\frac{\G(\l+s+1)}{2^{3s+2}\pi^{s+1}\G(\l+1)}\sum_{n=1}^{\infty}\sigma_s(n)\left(\mathcal{A}\left(s, \l, \sqrt{\a}, \sqrt{\b}, n\right)+\mathcal{A}\left(s, \l, \sqrt{\a}, \sqrt{\b}, -n\right)\right).
\end{align}
\end{theorem}
\begin{remark}
We do not know if there exist transformations analogous to the above for the remaining values of $s$, that is, the ones not satisfying \textup{Re}$(s+\l)>0$ or $s=-\l$. However, we do show by analytic continuation that the above result can be extended to \textup{Re}$(\l)>-1$. This is done in Section \ref{ac}. Of course, one can analytically continue Theorem \ref{anathmbdkz} to \textup{Re}$(\nu)>-m$, where $m\in\mathbb{N}$, however, we refrain ourselves from doing this as the resulting transformation becomes complicated.
\end{remark}
The above theorem as well as its analytic continuation given in Theorem \ref{anathmbdkzac} give many interesting corollaries. We state below the ones resulting from Theorem \ref{anathmbdkz} and reserve those resulting from Theorem \ref{anathmbdkzac} for Section \ref{ac}.
\begin{corollary}\label{dixmolcor}
Equation \eqref{dixmol} holds for \textup{Re}$(\b)>0$ and \textup{Re}$(\l)>-1$.
\end{corollary}

\noindent
Another important result is
\begin{corollary}\label{dixmolcorgen}
For \textup{Re}$(\b)>0$ and \textup{Re}$(s)>-\frac{1}{2}$,
\begin{align}\label{dixmolcorgeneqn}
&\sum_{n=1}^{\infty}\sigma_{s}(n)e^{-2\pi\sqrt{2n\b}}\cos\left(2\pi\sqrt{2n\b}\right)\nonumber\\
&=-\frac{1}{2}\zeta(-s)+\frac{\G\left(s+\frac{3}{2}\right)\zeta(-s)}{2\sqrt{\pi}(2\pi\b)^{s+1}}+\frac{\sqrt{\b}\G\left(s+\frac{3}{2}\right)}{2^{s+2}\pi^{s+\frac{3}{2}}}\sum_{n=1}^{\infty}\sigma_s(n)\left\{(\b+in)^{-s-\frac{3}{2}}+(\b-in)^{-s-\frac{3}{2}}\right\}.
\end{align}
\end{corollary}
The reason this identity is important because, specializing $s$ to be zero in the above identity gives an analogue of a result of Hardy \cite[Equation (2.12)]{hardyqjpam1915} which he used in his study of the Gauss circle problem to prove that
\begin{equation}\label{harob}
\sum_{n\leq x}r_2(n)-\pi x=\Omega(x^{1/4}).
\end{equation}
This analogue is
\begin{align}\label{dixmolcorgeneqns0}
\sum_{n=1}^{\infty}d(n)e^{-2\pi\sqrt{2n\b}}\cos\left(2\pi\sqrt{2n\b}\right)=\frac{1}{4}-\frac{1}{16\pi\b}+\frac{\sqrt{\b}}{8\pi}\sum_{n=1}^{\infty}d(n)\left\{(\b+in)^{-\frac{3}{2}}+(\b-in)^{-\frac{3}{2}}\right\}.
\end{align}
We end this section with a rather exotic transformation that results from Theorem \ref{anathmbdkz}.
\begin{corollary}\label{exotic}
For $|\arg \a|<\pi/2, |\arg \b|<\pi/2$ and \textup{Re}$(\sqrt{\a})>$ \textup{Re}$(\sqrt{\b})>0$,
\begin{align}\label{exoticeqn}
&\sum_{n=1}^{\infty}\sigma_{-1}(n)\bigg\{I_{1}\left(2\pi\overline{\epsilon}(\sqrt{n\a}-\sqrt{n\b})\right)K_{1}\left(2\pi\overline{\epsilon}(\sqrt{n\a}+\sqrt{n\b})\right)\nonumber\\
&\qquad\qquad\quad+I_{1}\left(2\pi\epsilon(\sqrt{n\a}-\sqrt{n\b})\right)K_{1}\left(2\pi\epsilon(\sqrt{n\a}+\sqrt{n\b})\right)\bigg\}\nonumber\\
&=-\frac{\pi^3(\a-\b)}{48}+\frac{1}{2}\left(\frac{\sqrt{\a}-\sqrt{\b}}{\sqrt{\a}+\sqrt{\b}}\right)\left\{\g+\log(2\pi\sqrt{\a\b})-1-\left(\frac{\sqrt{\a}+\sqrt{\b}}{\sqrt{\a}-\sqrt{\b}}\right)^2\log\left(\frac{4\sqrt{\a\b}}{(\sqrt{\a}+\sqrt{\b})^2}\right)\right\}\nonumber\\
&\quad+\frac{1}{2}\sum_{n=1}^{\infty}\sigma_{-1}(n)\left(\frac{\sqrt{\a+in}-\sqrt{\b+in}}{\sqrt{\a+in}+\sqrt{\b+in}}+\frac{\sqrt{\a-in}-\sqrt{\b-in}}{\sqrt{\a-in}+\sqrt{\b-in}}\right).
\end{align}
\end{corollary}

\noindent
That the special case $\l=1$ of \eqref{dixmol} results from the above corollary after dividing both sides by $\sqrt{\a}-\sqrt{\b}$ and then letting $\a\to\b$ can then be seen without much effort.

This paper is organized as follows. In Section \ref{ex}, we record the well-known asymptotic expansions of Bessel functions as their argument $\xi\to0$ or as $|\xi|\to\infty$. Assuming that the parameters involved in \eqref{grn693} and \eqref{koshcertain} are complex, we find the exact conditions on them so that these equations are valid. This is crucial in deriving Theorem \ref{tfockathm}. Sections \ref{proof1} and \ref{proof2} are devoted to the proofs of Theorems \ref{tfockathm} and \ref{anathmbdkz} respectively. Corollaries \ref{dixmolcor}, \ref{dixmolcorgen} and \ref{exotic} are proved in Section \ref{cor1}. Theorem \ref{anathmbdkz} admits analytic continuation to Re$(\nu)>-1$. This analytic continuation is obtained in Theorem \ref{anathmbdkzac} of Section \ref{ac}. The interesting corollaries that follow from this theorem, namely, Corollaries \ref{ramanujan-type} - \ref{soninetrans} are then derived. We end the paper with Section \ref{cr} consisting of concluding remarks and directions for further work.

\section{Extending the validity of certain integral evaluations}\label{ex}
Here we extend the validity of certain integral evaluations needed in the sequel by relaxing the conditions on the parameters involved. This is extremely crucial, for, most of the times we do require the widest possible conditions under which these formulas are valid, to prove our results. To do that, however, we need the following asymptotic formulas of the Bessel functions $J_{s}(\xi), Y_{s}(\xi)$, and $K_{s}(\xi)$, as $|\xi|\to\infty$, given by 
\cite[p.~199 and p.~202]{watson-1966a}
\begin{align}
J_{s}(\xi)&\sim \left(\frac{2}{\pi \xi}\right)^{\tf12}\bigg(\cos w\sum_{n=0}^{\infty}\frac{(-1)^n(s, 2n)}{(2\xi)^{2n}} -\sin w\sum_{n=0}^{\infty}\frac{(-1)^n(s, 2n+1)}{(2\xi)^{2n+1}}\bigg),\label{asymbess}\\
Y_{s}(\xi)&\sim \left(\frac{2}{\pi \xi}\right)^{\tf12}\bigg(\sin w\sum_{n=0}^{\infty}\frac{(-1)^n(s, 2n)}{(2\xi)^{2n}}+\cos w\sum_{n=0}^{\infty}\frac{(-1)^n(s, 2n+1)}{(2\xi)^{2n+1}}\bigg),\label{asymbess1}\\
K_{s}(\xi)&\sim \left(\frac{\pi}{2 \xi}\right)^{\tf12}e^{-\xi}\sum_{n=0}^{\infty}\frac{(s, n)}{(2\xi)^{n}},\label{asymbess2}
\end{align}
for $|\arg \xi|<\pi$. Here $w=\xi-\tfrac{1}{2}\pi s-\tfrac{1}{4}\pi$ and $(s,n)=\frac{\Gamma(s+n+1/2)}{\Gamma(n+1)\Gamma(s-n+1/2)}$.

Also from \cite[Equation (2.11)]{koshkernel},
\begin{equation}\label{fasy}
F_s(\rho t)<\!\!<_{s}\begin{cases}
\textup{max}(1, |\log(\rho t)|), &\mbox{ if } \quad s=0, 0 < \rho t \le 1, \\
(\rho t)^{-|\textup{Re}(s)|}, &\mbox{ if } \quad s\neq 0,  0 < \rho t\le 1, \\
				(\rho t)^{-1/2}, &\mbox{ if } \quad \rho t \ge 1.			
\end{cases}
\end{equation}
We begin with identity \eqref{grn693}. In \cite[p.~693, formula 6.596.7]{grn} as well as in \cite[p.~416, Equation (2)]{watson-1966a}, this result is stated to be valid for $a>0, \rho>0$, Re$(s)>-1$, and $|\arg\xi|<\pi/2$. However, in a footnote on page 416 in \cite{watson-1966a}, Watson says that with certain limitations, one can take $a$ and $\rho$ to be complex. Since the precise conditions on $a$ and $\rho$ are not given there and since we require the one for $a$ in order to prove our results (see \eqref{grn6931} and \eqref{grn6932} below), we give them here. 

Note that from \eqref{asymbess}, as $t\to\infty$,
\begin{equation*}
J_{s}(\rho t)=O\left(\sqrt{\frac{2}{\pi |\rho| t}}\left|\cos\left(\rho t-\frac{\pi s}{2}-\frac{\pi}{4}\right)\right|\right)=O_{s, \rho}\left(e^{|\textup{Im}(\rho)|t}\right),
\end{equation*}
where $|\arg(\rho)|=|\arg(\rho t)|<\pi$. Also, from \eqref{asymbess2},
\begin{equation*}
K_{\nu}\left(a\sqrt{t^2+\xi^2}\right)=O_{\nu, \xi}\left(e^{-\textup{Re}(a)t}\right)
\end{equation*}
as $t\to\infty$. Thus, to secure the convergence at the upper limit of integration, we require Re$(a)>|$\textup{Im}$(\rho)|$.

Now for complex values of $\rho, b, s$ and $\xi$, the precise conditions for the validity of \eqref{koshcertain} can be found to be Re$(\rho)>0$, $|\textup{Im}(\rho)|<$Re$(b\epsilon)$ and $|\textup{Im}(\rho)|<$Re$(b\overline{\epsilon})$, Re$(s)>-1$ and $|\arg(\xi)|<\frac{\pi}{4}$. However, since each of the variables $\rho, b$ and $\xi$ are positive in Theorem \ref{tfockathm}, we only show the validity of \eqref{koshcertain} for Re$(s)>-1$, which is required, for example, in obtaining \eqref{grn69312f}.

The condition Re$(s)>-1$ is needed to secure the convergence of the integral in \eqref{koshcertain} at the lower limit $0$. To see this, note that as $t\to 0$, we have \cite[p.~375, equations (9.6.9), (9.6.8)]{stab}
\begin{equation}\label{k0}
K_{s}(\rho t)\sim\begin{cases}
-\log(\rho t),\hspace{8mm}\text{if}\hspace{1mm}s=0,\\
\frac{1}{2}\G(s)\left(\frac{\rho t}{2}\right)^{-s},\hspace{1mm}\text{if}\hspace{1.5mm}\textup{Re}(s)>0,
\end{cases}
\end{equation}
and \cite[p.~360, equations (9.1.8), (9.1.9)]{stab}
\begin{equation}\label{y0}
Y_{s}(\rho t)\sim\begin{cases}
\frac{2}{\pi}\log(\rho t),\hspace{11mm}\text{if}\hspace{1mm}s=0,\\
-\frac{1}{\pi}\G(s)\left(\frac{\rho t}{2}\right)^{-s},\hspace{1mm}\text{if}\hspace{1.5mm}\textup{Re}(s)>0.
\end{cases}
\end{equation}
First let Re$(s)>0$. Note that
\begin{equation*}
\frac{K_{\nu}\left(b \sqrt{\xi^2+it^2}\right)}{(\xi^2+it^2)^{\mu/2}}+\frac{K_{\nu}\left(b \sqrt{\xi^2-it^2}\right)}{(\xi^2-it^2)^{\mu/2}}
\end{equation*}
approaches a constant depending on $b$, $\xi$ and $\mu$ as $t\to 0$, where $\xi\neq0$ if Re$(\nu)>0$. Also, from \eqref{k0} and \eqref{y0}, we deduce that
\begin{equation*}
M_{s}(\rho t)\sim\frac{2}{\pi}\G(s)\left(\frac{\rho t}{2}\right)^{-s}.
\end{equation*}
Since the integrand on the left-hand side of \eqref{koshcertain} contains $t^{s+1}$, convergence is clearly secured for Re$(s)>0$. In a similar way, it can be seen using \cite[Equation (2.14)]{dunster} and the corresponding asymptotic formulas for $J$ and $Y$-Bessel functions of purely imaginary order that the integral converges when Re$(s)=0$ too.

Now let $-1<$ Re$(s)<0$. Since the $K$-Bessel function is an even function of its order, we have from \eqref{k0},
\begin{equation}\label{kbigo}
\frac{2}{\pi}K_{s}(\rho t)=\frac{2}{\pi}K_{-s}(\rho t)\sim\frac{1}{\pi}\G(-s)\left(\frac{\rho t}{2}\right)^{s}.
\end{equation}
Unfortunately, the relation between $Y_{-s}(\rho t)$ and $Y_{s}(\rho t)$ is not as straightforward as in the case of $K_{-s}(\rho t)$ and $K_{s}(\rho t)$, but using the standard formulas in the theory of Bessel functions \cite[Formulas 8.482, 8.483, 8.484]{grn} it can be seen that
\begin{align}\label{ynu-nu}
Y_{s}(\rho t)=Y_{-s}(\rho t)\sec(\pi s)-J_{s}(\rho t)\tan(\pi s).
\end{align}
Now from \eqref{y0},
\begin{align}\label{y0-}
Y_{-s}(\rho t)\sim-\frac{1}{\pi}\G(-s)\left(\frac{\rho t}{2}\right)^{-s},
\end{align}
whereas from \cite[p.~360, Formula 9.1.7]{stab}, for $s\neq-1, -2, -3, \cdots$,
\begin{align}\label{j0}
J_{s}(\rho t)\sim\frac{\left(\rho t/2\right)^{s}}{\G(s+1)}
\end{align}
as $t\to 0$. Hence \eqref{ynu-nu}, \eqref{y0-} and \eqref{j0} give
\begin{equation}\label{ybigo}
Y_{s}(\rho t)=O_{s, \rho}\left(t^{\textup{Re}(s)}\right).
\end{equation}
Thus \eqref{kbigo} and \eqref{ybigo} give
\begin{equation*}
t^{s+1}M_{s}(\rho t)=O_{s, \rho}\left(t^{2\textup{Re}(s)+1}\right),
\end{equation*}
which implies that the convergence of the integral at $0$ is secured as long as Re$(s)>-1$. By \cite[p.~30, Theorem 2.3]{temme}, the integral in \eqref{koshcertain} represents a holomorphic function of $s$ for Re$(s)>-1$. Since the right-hand side is obviously analytic for Re$(s)>-1$, we conclude, by analytic continuation, that \eqref{koshcertain} is valid as long as Re$(s)>-1$.
%
%

\section{Proof of Theorem \ref{tfockathm}}\label{proof1}

We first prove the result for $z>w>0$ and later extend it by analytic continuation to complex $z$ and $w$ such that $|\arg z|<\pi/4, |\arg w|<\pi/4$ and \textup{Re}$(z)\geq$ \textup{Re}$(w)>0$. Also from the hypotheses of the theorem, we have $\rho>0$.

Letting $\xi=we^{-i\pi/4}$, and $a=y e^{i\pi/4}$ with $y>0$ in \eqref{grn693}, we see that for $\rho>0$ and Re$(s)>-1$, we have
\begin{align}\label{grn6931}
\int_{0}^{\infty}t^{s+1}J_{s}(\rho t)\frac{K_{\nu}\left(y\sqrt{w^2+it^2}\right)}{(w^2+it^2)^{\nu/2}}\, dt=\frac{\overline{\epsilon}^{(1+\nu+s)}\rho^{s}w^{1+s-\nu}}{y^{\nu}}\frac{K_{1+s-\nu}\left(w\overline{\epsilon}\sqrt{\rho^2+iy^2}\right)}{(\rho^2+ iy^2)^{\frac{1+s-\nu}{2}}}.
\end{align}
Similarly letting $\xi=we^{i\pi/4}$, and $a=y e^{-i\pi/4}$ with $y>0$, in \eqref{grn693}, we see that for $\rho>0$ and Re$(s)>-1$, we have
\begin{align}\label{grn6932}
\int_{0}^{\infty}t^{s+1}J_{s}(\rho t)\frac{K_{\nu}\left(y\sqrt{w^2-it^2}\right)}{(w^2-it^2)^{\nu/2}}\, dt=\frac{\epsilon^{1+\nu+s}\rho^{s}w^{1+s-\nu}}{y^{\nu}}\frac{K_{1+s-\nu}\left(w\epsilon\sqrt{\rho^2-iy^2}\right)}{(\rho^2- iy^2)^{\frac{1+s-\nu}{2}}}.
\end{align}
Adding \eqref{grn6931} and \eqref{grn6932}, we deduce that for $y>0$, $\rho>0$ and Re$(s)>-1$,
\begin{align}\label{grn69312}
&\int_{0}^{\infty}t^{s+1}J_{s}(\rho t)\left\{\frac{K_{\nu}\left(y\sqrt{w^2+it^2}\right)}{(w^2+it^2)^{\nu/2}}+\frac{K_{\nu}\left(y\sqrt{w^2-it^2}\right)}{(w^2-it^2)^{\nu/2}}\right\}\, dt\nonumber\\
&=\frac{w^{1+s-\nu}\rho^{s}}{y^{\nu}}\left\{\overline{\epsilon}^{\left(\nu+s+1\right)}\frac{K_{1+s-\nu}(w\overline{\epsilon} \sqrt{\rho^2+iy^2})}{(\rho^2+iy^2)^{\frac{1+s-\nu}{2}}}+\epsilon^{\nu+s+1}\frac{K_{1+s-\nu}(w\epsilon \sqrt{\rho^2-iy^2})}{(\rho^2-iy^2)^{\frac{1+s-\nu}{2}}}\right\}.
\end{align}
From \eqref{kk}, \eqref{grn69312}, and \eqref{koshcertain} with $\xi$ and $b$ replaced by $w$ and $y$ respectively, we see that for $y>0$, $\rho>0$ and Re$(s)>-1$,
\begin{align}\label{grn69312f}
&\int_{0}^{\infty}t^{s+1}F_{s}(\rho t)\left\{\frac{K_{\nu}\left(y\sqrt{w^2+it^2}\right)}{(w^2+it^2)^{\nu/2}}+\frac{K_{\nu}\left(y\sqrt{w^2-it^2}\right)}{(w^2-it^2)^{\nu/2}}\right\}\, dt\nonumber\\
&=\frac{w^{1+s-\nu}\rho^{s}}{y^{\nu}}\left\{\frac{K_{1+s-\nu}(w\sqrt{y^2-i\rho^2})}{(y^2-i\rho^2)^{\frac{1+s-\nu}{2}}}+\frac{K_{1+s-\nu}(w\sqrt{y^2+i\rho^2})}{(y^2+i\rho^2)^{\frac{1+s-\nu}{2}}}\right\}.
\end{align}
Now let $\nu=s+1/2$ in the above equation. Then, along with the fact \cite[p.~925, Formula \textbf{8.469.3}]{grn} that 
\begin{equation}\label{khalf}
K_{\pm\frac{1}{2}}(x)=\sqrt{\frac{\pi}{2x}}e^{-x},
\end{equation}
we deduce that for $y>0$, $\rho>0$ and Re$(s)>-1$,
\begin{align*}
&y^{s+\frac{1}{2}}\int_{0}^{\infty}t^{s+1}F_{s}(\rho t)\left\{\frac{K_{s+\frac{1}{2}}\left(y\sqrt{w^2+it^2}\right)}{(w^2+it^2)^{\frac{s}{2}+\frac{1}{4}}}+\frac{K_{s+\frac{1}{2}}\left(y\sqrt{w^2-it^2}\right)}{(w^2-it^2)^{\frac{s}{2}+\frac{1}{4}}}\right\}\, dt\nonumber\\
&=\sqrt{\frac{\pi}{2}}\rho^{s}\left\{\frac{\textup{exp}\left(-w\sqrt{y^2+(\rho\epsilon)^2}\right)}{\sqrt{y^2+(\rho\epsilon)^2}}+\frac{\textup{exp}\left(-w\sqrt{y^2+(\rho\overline{\epsilon})^2}\right)}{\sqrt{y^2+(\rho\overline{\epsilon})^2}}\right\}.
\end{align*}
Now multiply both sides of the above equation by $J_{2\l}\left(y\sqrt{z^2-w^2}\right)$, where Re$(\l)>-1/2$, integrate both sides of the resulting equation with respect to $y$ from $0$ to $\infty$. The double integral on the left is absolutely convergent as can be seen from \eqref{asymbess}, \eqref{asymbess2}, \eqref{fasy}, \eqref{k0} and \eqref{j0}. Hence interchanging the order of integration we see that for $\rho>0$, Re$(s)>-1$ and Re$(\l)>-1/2$,
\begin{align}\label{tfocka}
&\int_{0}^{\infty}t^{s+1}F_{s}(\rho t)\int_{0}^{\infty}y^{s+\frac{1}{2}}J_{2\l}\left(y\sqrt{z^2-w^2}\right)\left\{\frac{K_{s+\frac{1}{2}}\left(y\sqrt{w^2+it^2}\right)}{(w^2+it^2)^{\frac{s}{2}+\frac{1}{4}}}+\frac{K_{s+\frac{1}{2}}\left(y\sqrt{w^2-it^2}\right)}{(w^2-it^2)^{\frac{s}{2}+\frac{1}{4}}}\right\}\, dy\, dt\nonumber\\
&=\sqrt{\frac{\pi}{2}}\rho^{s}\int_{0}^{\infty}J_{2\l}\left(y\sqrt{z^2-w^2}\right)\left\{\frac{\textup{exp}\left(-w\sqrt{y^2+(\rho\epsilon)^2}\right)}{\sqrt{y^2+(\rho\epsilon)^2}}+\frac{\textup{exp}\left(-w\sqrt{y^2+(\rho\overline{\epsilon})^2}\right)}{\sqrt{y^2+(\rho\overline{\epsilon})^2}}\right\}\, dy.
\end{align}
Our next task is to evaluate the integral on the right-hand side of the above equation. To that end, note that from \cite[p.~708, formula (6.637.1)]{grn}, we have for Re $\a>0$, Re $\b>0, \g>0$, and Re $\nu>-1$,
\begin{equation}\label{grn708-1}
\int_{0}^{\infty}J_{\nu}(\g x)\frac{\textup{exp}\left(-\a\sqrt{x^2+\b^2}\right)}{\sqrt{x^2+\b^2}}\, dx=I_{\frac{\nu}{2}}\left(\frac{\b}{2}\left(\sqrt{\a^2+\g^2}-\a\right)\right)K_{\frac{\nu}{2}}\left(\frac{\b}{2}\left(\sqrt{\a^2+\g^2}+\a\right)\right).
\end{equation}
Now let $\a=w, \g=\sqrt{z^2-w^2}$ and $\nu=2\l$ in \eqref{grn708-1}. Then replace $\b$ first by $\rho\epsilon$ and then by $\rho\overline{\epsilon}$ and add the resulting equations to see that for $\rho>0$ and Re$(\l)>-1/2$,
\begin{align}\label{p11}
&\int_{0}^{\infty}J_{2\l}\left(y\sqrt{z^2-w^2}\right)\left(\frac{\textup{exp}\left(-w\sqrt{(\rho\overline{\epsilon})^2+y^2}\right)}{\sqrt{(\rho\overline{\epsilon})^2+y^2}}+\frac{\textup{exp}\left(-w\sqrt{(\rho\epsilon)^2+y^2}\right)}{\sqrt{(\rho\epsilon)^2+y^2}}\right)\, dy\nonumber\\
&=I_{\l}\left(\frac{\rho\overline{\epsilon}(z-w)}{2}\right)K_{\l}\left(\frac{\rho\overline{\epsilon}(z+w)}{2}\right)+I_{\l}\left(\frac{\rho\epsilon(z-w)}{2}\right)K_{\l}\left(\frac{\rho\epsilon(z+w)}{2}\right).
\end{align}
From \eqref{tfocka} and \eqref{p11}, we have for $\rho>0$, Re$(s)>-1$ and Re$(\l)>-1/2$,
\begin{align}\label{tfocka1}
&\int_{0}^{\infty}t^{s+1}F_{s}(\rho t)\int_{0}^{\infty}y^{s+\frac{1}{2}}J_{2\l}\left(y\sqrt{z^2-w^2}\right)\left\{\frac{K_{s+\frac{1}{2}}\left(y\sqrt{w^2+it^2}\right)}{(w^2+it^2)^{\frac{s}{2}+\frac{1}{4}}}+\frac{K_{s+\frac{1}{2}}\left(y\sqrt{w^2-it^2}\right)}{(w^2-it^2)^{\frac{s}{2}+\frac{1}{4}}}\right\}\, dy\, dt\nonumber\\
&=\sqrt{\frac{\pi}{2}}\rho^{s}\left\{I_{\l}\left(\frac{\rho\overline{\epsilon}(z-w)}{2}\right)K_{\l}\left(\frac{\rho\overline{\epsilon}(z+w)}{2}\right)+I_{\l}\left(\frac{\rho\epsilon(z-w)}{2}\right)K_{\l}\left(\frac{\rho\epsilon(z+w)}{2}\right)\right\}.
\end{align}
To evaluate the inner integral on the left side of \eqref{tfocka1}, we apply the following formula from \cite[p.~365, formula \textbf{2.16.21.1}]{pbm}, valid for Re$(c)>|$Im$(b)|$ and Re$(\a+u)>|$Re$(v)|$:
\begin{align*}
\int_{0}^{\infty}x^{\a-1}J_{u}(bx)K_{v}(cx)\, dx=2^{\a-2}b^{u}c^{-\a-u}\frac{\G\left(\frac{\a+u+v}{2}\right)\G\left(\frac{\a+u-v}{2}\right)}{\G(u+1)}\pFq21{\frac{\a+u+v}{2}, \frac{\a+u-v}{2}}{u+1}{-\frac{b^2}{c^2}}.
\end{align*}
Now replace $\a$ by $s+3/2$, $u$ by $2\l$, $v$ by $s+1/2$, $b$ by $\sqrt{z^2-w^2}$ and $c$ by $\sqrt{w^2+it^2}$ in the above equation so that for Re$(\l)>-1/2$ and Re$(s+\l)>-1$, we have
{\allowdisplaybreaks\begin{align}\label{prudev1}
&\int_{0}^{\infty}y^{s+\frac{1}{2}}J_{2\l}\left(y\sqrt{z^2-w^2}\right)\frac{K_{s+\frac{1}{2}}\left(y\sqrt{w^2+it^2}\right)}{(w^2+it^2)^{\frac{s}{2}+\frac{1}{4}}}\, dy\nonumber\\
&=\frac{1}{(w^2+it^2)^{\frac{s}{2}+\frac{1}{4}}}\bigg\{2^{s-\frac{1}{2}}(z^2-w^2)^{\l}(w^2+it^2)^{-\frac{1}{2}\left(s+\frac{3}{2}+2\l\right)}\frac{\G(s+\l+1)\G\left(\l+\frac{1}{2}\right)}{\G(2\l+1)}\nonumber\\
&\quad\quad\quad\quad\quad\quad\quad\quad\times\pFq21{s+\l+1, \l+\frac{1}{2}}{2\l+1}{\frac{w^2-z^2}{w^2+it^2}}\bigg\}.
\end{align}}
We first appropriately transform the ${}_2F_{1}$ in the above equation. By Pfaff's transformation \cite[p.~110, Equation (5.5)]{temme}, for $|\arg(1-\xi)|<\pi$, we have
\begin{equation*}
_2F_1\left(\begin{matrix} a,b\\c\end{matrix}\,\Bigr|\, \xi\right)=(1-\xi)^{-b}{}_2F_1\left(\begin{matrix} c-a,b\\c\end{matrix}\,\Bigr|\, \frac{\xi}{\xi-1}\right).
\end{equation*}
Hence
\begin{equation}\label{2f1pfaff}
\pFq21{s+\l+1, \l+\frac{1}{2}}{2\l+1}{\frac{w^2-z^2}{w^2+it^2}}=\left(\frac{z^2+it^2}{w^2+it^2}\right)^{-\l-\frac{1}{2}}\pFq21{\l-s, \l+\frac{1}{2}}{2\l+1}{\frac{z^2-w^2}{z^2+it^2}}.
\end{equation}
Now for $|\xi|<1$, we have \cite[p.~130, Problem 5.7, Equation (3)]{temme},
\begin{equation}\label{temmehyp}
{}_2F_{1}\left(\begin{matrix} a,b\\2b\end{matrix}\,\Bigr|\, \xi\right)=\left(\frac{1+\sqrt{1-\xi}}{2}\right)^{-2a}{}_2F_{1}\left(\begin{matrix} a,a-b+\frac{1}{2}\\b+\frac{1}{2}\end{matrix}\,\Bigr|\, \left(\frac{1-\sqrt{1-\xi}}{1+\sqrt{1-\xi}}\right)^2\right).
\end{equation}
Since $z>w>0$, we have $|z^2-w^2|<|z^2+it^2|$ for all $t>0$. Hence from \eqref{temmehyp},
\begin{align}\label{2f1qt}
\pFq21{\l-s, \l+\frac{1}{2}}{2\l+1}{\frac{z^2-w^2}{z^2+it^2}}&=\frac{2^{2\l-2s}(z^2+it^2)^{\l-s}}{\left(\sqrt{z^2+it^2}+\sqrt{w^2+it^2}\right)^{2\l-2s}}\nonumber\\
&\quad\times\pFq21{\l-s, -s}{\l+1}{\left(\frac{\sqrt{z^2+it^2}-\sqrt{w^2+it^2}}{\sqrt{z^2+it^2}+\sqrt{w^2+it^2}}\right)^2}.
\end{align}
Substitute \eqref{2f1qt} in \eqref{2f1pfaff} to get
\begin{align}\label{2f1pfaffqt}
\pFq21{s+\l+1, \l+\frac{1}{2}}{2\l+1}{\frac{w^2-z^2}{w^2+it^2}}&=\frac{2^{2\l-2s}}{\left(\sqrt{z^2+it^2}+\sqrt{w^2+it^2}\right)^{2\l-2s}}\frac{(z^2+it^2)^{-s-\frac{1}{2}}}{(w^2+it^2)^{-\l-\frac{1}{2}}}\nonumber\\
&\quad\times\pFq21{\l-s, -s}{\l+1}{\left(\frac{\sqrt{z^2+it^2}-\sqrt{w^2+it^2}}{\sqrt{z^2+it^2}+\sqrt{w^2+it^2}}\right)^2}.
\end{align}
Next, substituting \eqref{2f1pfaffqt} in \eqref{prudev1} gives for Re$(\l)>-1/2$ and Re$(s+\l)>-1$,
\begin{align}\label{prudev1s}
\int_{0}^{\infty}y^{s+\frac{1}{2}}J_{2\l}\left(y\sqrt{z^2-w^2}\right)\frac{K_{s+\frac{1}{2}}\left(y\sqrt{w^2+it^2}\right)}{(w^2+it^2)^{\frac{s}{2}+\frac{1}{4}}}\, dy
=\frac{2^{-s-\frac{1}{2}}\sqrt{\pi}\G(s+\l+1)}{\G(\l+1)}\mathcal{A}(t^2),
\end{align}
upon simplification, where $\mathcal{A}(t)$ is defined in \eqref{funa}.

Similarly for Re$(\l)>-1/2$ and Re$(s+\l)>-1$, we have
\begin{align}\label{prudev2s}
\int_{0}^{\infty}y^{s+\frac{1}{2}}J_{2\l}\left(y\sqrt{z^2-w^2}\right)\frac{K_{s+\frac{1}{2}}\left(y\sqrt{w^2-it^2}\right)}{(w^2-it^2)^{\frac{s}{2}+\frac{1}{4}}}\, dy=\frac{2^{-s-\frac{1}{2}}\sqrt{\pi}\G(s+\l+1)}{\G(\l+1)}\mathcal{A}(-t^2).
\end{align}
Add the corresponding sides of \eqref{prudev1s} and \eqref{prudev2s} and substitute the right-hand side of the resulting equation for the inner integral in \eqref{tfocka1} so as to deduce \eqref{tfockaeqn} for $z>w>0$ after simplification.

Now consider $z$ and $w$ to be complex numbers. The integrand of the left-hand side of \eqref{tfockaeqn} is well-defined except for the singularities on the lines $\arg z=\pm\frac{\pi}{4}$ and $\arg w=\pm\frac{\pi}{4}$. The right-hand side is well-defined for all complex $z, w$ such that $z\neq\pm w$. Therefore, \eqref{tfockaeqn} can be analytically continued in the region $|\arg z|<\pi/4, |\arg w|<\pi/4$ and \textup{Re}$(z)>$ \textup{Re}$(w)>0$.

\section{Proof of Theorem \ref{anathmbdkz}}\label{proof2}
We first prove the result for $\a>\b>0$ and later extend it by analytic continuation to the region \textup{Re}$(\sqrt{\a})\geq$ \textup{Re}$(\sqrt{\b})>0$. Note that $s, \l, \a$ and $\b$ are fixed throughout the proof.
In Theorem \ref{tfockathm}, let $\rho=4\pi\sqrt{x}$, $x>0$, replace $z, w$ and $t$ by $\sqrt{\a}, \sqrt{\b}$ and $\sqrt{t}$ respectively. Define 
\begin{equation}\label{f}
f(t):=\frac{1}{2\pi}t^{\frac{s}{2}}\left(\mathcal{A}(t)+\mathcal{A}(-t))\right),
\end{equation}
where $\mathcal{A}(t)$ is the same as defined in \eqref{funa} except that we replace  $z$ and $w$ by $\sqrt{\a}$ and $\sqrt{\b}$ respectively.

Then from Theorem \ref{tfockathm}, for $x>0$, Re$(s)>-1$, Re$(\l)>-1/2$ and Re$(s+\l)>-1$,
\begin{align}\label{g0}
2\pi \int_{0}^{\infty} f(t) F_s(4\pi\sqrt{xt}) \,dt&=\frac{2\G(\l+1)(8\pi\sqrt{x})^{s}}{\G(s+\l+1)}\bigg(I_{\l}\left(2\pi\overline{\epsilon}\sqrt{x}(\sqrt{\a}-\sqrt{\b})\right)K_{\l}\left(2\pi\overline{\epsilon}\sqrt{x}(\sqrt{\a}+\sqrt{\b})\right)\nonumber\\
&\quad+I_{\l}\left(2\pi\epsilon\sqrt{x}(\sqrt{\a}-\sqrt{\b})\right)K_{\l}\left(2\pi\epsilon\sqrt{x}(\sqrt{\a}+\sqrt{\b})\right)\bigg).
\end{align}
The function $f$ defined in \eqref{f} satisfies the conditions given before \eqref{recpairgui} so that Guinand's result \eqref{guitra} could be applied. To see this, first note $\a>\b>0$ and Re$(\l)>0$. Since $f$ is infinitely differentiable, $f$ and $f'$ are integrals. Also, as $|\xi|\to \infty$, one has \cite[Equation (4.4)]{bdkz}
\begin{align}\label{asssy}
\mathcal{A}(-i\xi)\sim \frac{4^{s-\l}(\alpha-\beta)^{\l}}{\xi^{\l+s+1}},
\end{align}
since the hypergeometric function tends to $1$. By the hypotheses of Theorem \ref{anathmbdkz}, we have Re$(\l)>0$, Re$(s+\l)>0$ or $s=-\l$. Hence we conclude that $f$ tends to zero as $t\to\infty$. Further, doing the analysis similar to that done in \cite[p.~322]{bdkz} for $r_k(n)$, one can prove that $f(t)$, $tf'(t)$ and $t^2f''(t)$ belong to $L^{2}(0,\infty)$. Thus the hypotheses for applying Guinand's result \eqref{guitra} are satisfied.

Note that \eqref{recpairgui} and \eqref{g0} imply
\begin{align}\label{g}
g(x)&=\frac{2\G(\l+1)(8\pi\sqrt{x})^{s}}{\G(s+\l+1)}\bigg\{I_{\l}\left(2\pi\overline{\epsilon}\sqrt{x}(\sqrt{\a}-\sqrt{\b})\right)K_{\l}\left(2\pi\overline{\epsilon}\sqrt{x}(\sqrt{\a}+\sqrt{\b})\right)\nonumber\\
&\quad+I_{\l}\left(2\pi\epsilon\sqrt{x}(\sqrt{\a}-\sqrt{\b})\right)K_{\l}\left(2\pi\epsilon\sqrt{x}(\sqrt{\a}+\sqrt{\b})\right)\bigg\}.
\end{align}
Next, we need to evaluate the integrals
\begin{equation*}
\int_{0}^{\infty} y^{\pm\frac{s}{2}} f(y) \,dy\hspace{5mm}\text{and}\hspace{5mm}\int_{0}^{\infty} y^{\pm\frac{s}{2}} g(y) \,dy.
\end{equation*}
To that end, let $y\geq 0$. 
Then for Re$(\mu+\l+1)>0$, we have \cite[Equation (4.13)]{bdkz},
 \begin{align}\label{itransform}
\int_{0}^{\infty}t^{\mu}I_{\lambda}(t)e^{-\frac{(2y+\alpha+\beta)t}{\alpha-\beta}}\,dt=&\frac{2^{-3\mu-1}}{\sqrt{y+\alpha}\sqrt{y+\beta}}\frac{(\alpha-\beta)^{\lambda+\mu+1}\Gamma(\lambda+\mu+1)}{(\sqrt{y+\alpha}+\sqrt{y+\beta})^{2\lambda}\Gamma(\lambda+1)}\left(\frac{1}{\sqrt{y+\alpha}}+\frac{1}{\sqrt{y+\beta}}\right)^{2\mu}\nonumber\\
                               &  \times    \pFq21{\l-\mu, -\mu}{\l+1}{\left(\frac{\sqrt{y+\alpha}-\sqrt{y+\beta}}{\sqrt{y+\alpha}+\sqrt{y+\beta}}\right)^2}.
\end{align}
Let $\nu$ be a complex number such that Re$(\nu)>-1$. Then from \eqref{itransform}, we see that for Re$(s+\l+1)>0$,
{\allowdisplaybreaks\begin{align}\label{doublei}
\int_{0}^{\infty}y^{\nu}\mathcal{A}(-iy)\, dy&=\frac{2^{3s+1}(\alpha-\beta)^{-s-1}\Gamma(\lambda+1)}{\Gamma(\l+s+1)}\int_0^{\infty}\int_0^{\infty}y^{\nu}t^se^{-\frac{(2y+\alpha+\beta)t}{\alpha-\beta}}I_{\l}(t)\,dt\,dy\nonumber\\
&=\frac{2^{3s+1}(\alpha-\beta)^{-s-1}\Gamma(\lambda+1)}{\Gamma(\l+s+1)}\int_0^{\infty}t^sI_{\l}(t)e^{-\frac{(\alpha+\beta)t}{\alpha-\beta}}\left(\int_0^{\infty}y^{\nu}e^{-\frac{2yt}{\alpha-\beta}}\,dy\right)\,dt,
\end{align}}%
where the interchange of the order of integration is justified since the double integral in the first line is absolutely convergent, for, as $t\to\infty$, we have \cite[p.~240, Equation (9.54)]{temme}
\begin{align}\label{iasy}
I_{\l}(t)\sim \frac{e^t}{\sqrt{2\pi t}},
\end{align}
where Re$(\l)>0$ and $|\l|\to\infty$. Since Re$(\nu)>-1$, expressing the inner integral on the extreme right of \eqref{doublei} in the form of a gamma function, we have for Re$(\l+s-\nu)>0$,
\begin{align}\label{doubleis}
\int_{0}^{\infty}y^{\nu}\mathcal{A}(-iy)\, dy&=\frac{2^{3s-\nu}(\alpha-\beta)^{\nu-s}\Gamma(\lambda+1)\Gamma(\nu+1)}{\Gamma(\l+s+1)}\int_0^{\infty}t^{s-\nu-1}I_{\l}(t)e^{-\frac{(\alpha+\beta)t}{\alpha-\beta}}\,dt\nonumber\\
&=\frac{4^{\nu+1}\Gamma(\lambda+s-\nu)\Gamma(\nu+1)}{\Gamma(\l+s+1)}\left(\frac{\sqrt{\alpha}-\sqrt{\beta}}{\sqrt{\alpha}+\sqrt{\beta}}\right)^{\l}\frac{1}{\sqrt{\alpha\beta}}\left(\frac{1}{\sqrt{\alpha}}+\frac{1}{\sqrt{\beta}}\right)^{2(s-\nu-1)}\\\nonumber
	&\quad\times\pFq21{\l-s+\nu+1, -s+\nu+1}{\l+1}{\left(\frac{\sqrt{\alpha}-\sqrt{\beta}}{\sqrt{\alpha}+\sqrt{\beta}}\right)^2},
\end{align}
where in the last step, we again employed \eqref{itransform}, this time with $y=0$ and $\mu=s-\nu-1$. 

Let $C(\theta):=C(R, \theta):=\big\{\xi=Re^{i\theta}: 0\leq \theta\leq \frac{\pi}{2}, R>0\big\}$ be a circular arc in the first quadrant. Then $C(-\theta)$ is a circular arc in the fourth quadrant. From \eqref{asssy}, clearly, 
\begin{align*}
\int_{C(\theta)}\xi^{\nu}\mathcal{A}(-i\xi)\, d\xi,\quad \int_{C(-\theta)}\xi^{\nu}\mathcal{A}(-i\xi) d\xi\ll R^{-\text{Re }(s+\l-\nu+1)},
\end{align*} 
and the latter tends to zero as $R\to \infty$ since Re$(\l+s-\nu)>0$. Note that $\xi^{\nu}\mathcal{A}(\xi)$ is an analytic function for $\text{Re}(\xi)\geq 0$. Thus invoking Cauchy's residue theorem twice, once along the closed contour consisting of $[0, R]$, $C(\theta)$ and $[iR, 0]$, and then, along its reflection along the real axis, we find that
\begin{align*}
\int_0^{i\infty}\xi^{\nu}\mathcal{A}(-i\xi)\, d\xi=\int_0^{\infty}\xi^{\nu}\mathcal{A}(-i\xi)\, d\xi=-\int_{-i\infty}^0\xi^{\nu}\mathcal{A}(-i\xi)\, d\xi.
\end{align*}
Therefore, 
\begin{align}\label{hintfirst}
\int_0^{\infty}y^{\nu}\mathcal{A}(y)\, dy=i^{-\nu-1}\int_0^{\infty}y^{\nu}\mathcal{A}(-iy)\, dy
\end{align}
and
\begin{align}\label{hintfourth}
\int_0^{\infty}y^{\nu}\mathcal{A}(-y)\, dy=i^{\nu+1}\int_0^{\infty}y^{\nu}\mathcal{A}(-iy)\, dy,
\end{align}
so that from \eqref{doubleis}, \eqref{hintfirst}, and \eqref{hintfourth}, we have for $\text{Re}(\nu)>-1$ and $\text{Re}(s+\l-\nu)>0$, 
\begin{align}\label{nuinte}
\int_{0}^{\infty}y^{\nu}(\mathcal{A}(y)+\mathcal{A}(-y))\, dy&=-\frac{2^{2\nu+3}\sin\left(\frac{\pi\nu}{2}\right)\Gamma(\lambda+s-\nu)\Gamma(\nu+1)}{\sqrt{\alpha\beta}\Gamma(\l+s+1)}\left(\frac{1}{\sqrt{\alpha}}+\frac{1}{\sqrt{\beta}}\right)^{2(s-\nu-1)}\nonumber\\
&\quad\times\left(\frac{\sqrt{\alpha}-\sqrt{\beta}}{\sqrt{\alpha}+\sqrt{\beta}}\right)^{\l}\pFq21{\l-s+\nu+1, -s+\nu+1}{\l+1}{\left(\frac{\sqrt{\alpha}-\sqrt{\beta}}{\sqrt{\alpha}+\sqrt{\beta}}\right)^2}.
\end{align}
We are now ready to evaluate $\int_{0}^{\infty} y^{\pm\frac{s}{2}} f(y) \,dy$. To that end, let $\nu=s$ in \eqref{nuinte} so that for Re$(s)>-1$ and Re$(\l)>0$, we have using \eqref{f},
\begin{align}\label{f1}
\int_{0}^{\infty} y^{\frac{s}{2}} f(y) \,dy=-\frac{2^{2s}\Gamma(\l)\Gamma(s+1)\sin\left(\frac{\pi s}{2}\right)}{\pi\Gamma(\l+s+1)}\left(\frac{\sqrt{\alpha}-\sqrt{\beta}}{\sqrt{\alpha}+\sqrt{\beta}}\right)^{\l}.
\end{align}
Now first let $\nu=0$ in \eqref{nuinte} so that for Re$(s+\l)>0$,
\begin{align}\label{f2}
\int_{0}^{\infty} y^{-\frac{s}{2}} f(y) \,dy=0.
\end{align}
We now evaluate the integral in \eqref{f2} when $s=-\l$. Note that in this case, \eqref{nuinte} implies that for $-1<$ Re$(\nu)<0$,
\begin{align}\label{nuinte1}
\int_{0}^{\infty}y^{\nu}(\mathcal{A}(y)+\mathcal{A}(-y))\, dy&=\frac{\pi 2^{2\nu+2}(\a\b)^{\l+\nu+1/2}(\sqrt{\a}-\sqrt{\b})^{\l}}{\cos\left(\frac{\pi\nu}{2}\right)(\sqrt{\a}+\sqrt{\b})^{3\l+2\nu+2}}\nonumber\\
&\quad\times\pFq21{2\l+\nu+1, \l+\nu+1}{\l+1}{\left(\frac{\sqrt{\alpha}-\sqrt{\beta}}{\sqrt{\alpha}+\sqrt{\beta}}\right)^2},
\end{align}
where we used a variant of the reflection formula for the gamma function \cite[p.~74]{temme}, namely,
\begin{equation*}
\G\left(\frac{1}{2}+\theta\right)\G\left(\frac{1}{2}-\theta\right)=\frac{\pi}{\cos(\pi\theta)},\hspace{1.5mm}\text{where}\hspace{1.5mm}\theta-\frac{1}{2}\notin\mathbb{Z}.
\end{equation*}
Note that both sides of \eqref{nuinte1} are analytic at $\nu=0$, hence letting $\nu=0$ in \eqref{nuinte1} gives
\begin{align}\label{f21}
\int_{0}^{\infty} y^{-\frac{s}{2}} f(y) \,dy&=\frac{2(\a\b)^{\l+\frac{1}{2}}(\sqrt{\a}-\sqrt{\b})^{\l}}{(\sqrt{\a}+\sqrt{\b})^{3\l+2}}\pFq10{2\l+1}{-}{\left(\frac{\sqrt{\alpha}-\sqrt{\beta}}{\sqrt{\alpha}+\sqrt{\beta}}\right)^2}\nonumber\\
&=\frac{2(\a\b)^{\l+\frac{1}{2}}(\sqrt{\a}-\sqrt{\b})^{\l}}{(\sqrt{\a}+\sqrt{\b})^{3\l+2}}\sum_{n=0}^{\infty}\frac{(2\l+1)_n}{n!}\left(\frac{\sqrt{\alpha}-\sqrt{\beta}}{\sqrt{\alpha}+\sqrt{\beta}}\right)^{2n}\nonumber\\
&=\frac{2(\a\b)^{\l+\frac{1}{2}}(\sqrt{\a}-\sqrt{\b})^{\l}}{(\sqrt{\a}+\sqrt{\b})^{3\l+2}}\frac{(4\sqrt{\a\b})^{-2\l-1}}{(\sqrt{\a}+\sqrt{\b})^{-4\l-2}}\nonumber\\
&=\frac{(\a-\b)^{\l}}{2^{4\l+1}},
\end{align}%
where in the third step, we employed the generalized binomial theorem \cite[p.~108, Equation (5.2)]{temme}, namely, for $|\xi|<1$,
\begin{equation*}
(1-\xi)^{-a}=\sum_{n=0}^{\infty}\frac{(a)_n}{n!}\xi^{n}.
\end{equation*}
Combining \eqref{f2} and \eqref{f21} and using the definition of $h(s, \l)$ in \eqref{hsl}, we can write
\begin{equation}\label{fboth}
\int_{0}^{\infty} y^{-\frac{s}{2}} f(y) \,dy=h(s,\l).
\end{equation}
To evaluate the remaining two integrals, namely, $\int_{0}^{\infty} y^{\pm\frac{s}{2}} g(y) \,dy$, we first note from \cite[Equation (24)]{dixfer4}\footnote{As mentioned in \cite{dixfer4}, the first equality in \eqref{xik} is an obvious modification of \cite[p.~410, Equation (13.45)]{watson-1966a}.} (see also \cite[p.~380-381, Equation \textbf{2.16.28.1}]{pbm}) that for $|\textup{Re}(b)|<$Re$(c)$ and $|\textup{Re}(\nu)|<$ Re$(a+\nu)$,
\begin{align}\label{xik}
\int_0^{\infty}y^{a-1}I_{\nu}(by)K_{\nu}(cy)\, dy
&=2^{a-2}c^{-a-\nu}b^{\nu}\frac{\G\left(a/2\right)\G\left(\nu+a/2\right)}{\G(\nu+1)}{}_2F_{1}\left(\begin{matrix} \nu+\frac{a}{2},\frac{a}{2}\\\nu+1\end{matrix}\,\Bigr|\, \frac{b^2}{c^2}\right)\nonumber\\
&=\frac{(2c)^{a-2}(b/c)^{\nu}\Gamma(a/2)\Gamma(\nu+a/2)}{(c^2-b^2)^{a-1}\Gamma(\nu+1)}\pFq21{1+\nu-a/2, 1-a/2}{1+\nu}{\frac{b^2}{c^2}},
\end{align}
where the second equality follows from Euler's transformation \cite[p.~110, Equation (5.5)]{temme}
\begin{equation*}
_2F_1\left(\begin{matrix} a,b\\c\end{matrix}\,\Bigr|\, \xi\right)=(1-\xi)^{c-a-b}{}_2F_1\left(\begin{matrix} c-a,c-b\\c\end{matrix}\,\Bigr|\, \xi\right).
\end{equation*}
In \eqref{xik}, replace $y$ by $\sqrt{y}$ and let $a=2(s+1)$, $b=2\pi\bar{\epsilon}(\sqrt{\alpha}-\sqrt{\beta})$, $c=2\pi\bar{\epsilon}(\sqrt{\alpha}+\sqrt{\beta})$ and $\nu=\lambda$ so that for $\text{Re}(s)>-1$ and $\text{Re}(s+\l)>-1$,
\begin{align}\label{xik1}
&\int_0^{\infty}y^{s}I_{\l}(2\pi\bar{\epsilon}\sqrt{y(\alpha-\beta)})K_{\l}(2\pi\bar{\epsilon}\sqrt{y(\alpha+\beta)})\, dy\\\nonumber
&=\frac{i^{s+1}\Gamma(\lambda+s+1)\Gamma(s+1)}{\sqrt{\alpha\beta}(2\pi)^{2s+2}2^{2s+1}\Gamma(\l+1)}\left(\frac{\sqrt{\alpha}-\sqrt{\beta}}{\sqrt{\alpha}+\sqrt{\beta}}\right)^{\l}\left(\frac{1}{\sqrt{\alpha}}+\frac{1}{\sqrt{\beta}}\right)^{2s}\pFq21{\l-s, -s}{\l+1}{\left(\frac{\sqrt{\alpha}-\sqrt{\beta}}{\sqrt{\alpha}+\sqrt{\beta}}\right)^2}.
\end{align}
Similarly for $\text{Re}(s)>-1$ and $\text{Re}(s+\l)>-1$,
\begin{align}\label{xik2}
&\int_0^{\infty}y^{s}I_{\l}(2\pi{\epsilon}\sqrt{y(\alpha-\beta)})K_{\l}(2\pi{\epsilon}\sqrt{y(\alpha+\beta)})\, dy\\\nonumber
&=\frac{(-i)^{s+1}\Gamma(\lambda+s+1)\Gamma(s+1)}{\sqrt{\alpha\beta}(2\pi)^{2s+2}2^{2s+1}\Gamma(\l+1)}\left(\frac{\sqrt{\alpha}-\sqrt{\beta}}{\sqrt{\alpha}+\sqrt{\beta}}\right)^{\l}\left(\frac{1}{\sqrt{\alpha}}+\frac{1}{\sqrt{\beta}}\right)^{2s}\pFq21{\l-s, -s}{\l+1}{\left(\frac{\sqrt{\alpha}-\sqrt{\beta}}{\sqrt{\alpha}+\sqrt{\beta}}\right)^2}.
\end{align}
From \eqref{g}, \eqref{xik1} and \eqref{xik2}, we deduce that for $\text{Re}(s)>-1$ and $\text{Re}(s+\l)>-1$,
\begin{align}\label{xik3}
\int_0^{\infty}y^{\frac{s}{2}}g(y)\, dy&=-\frac{(2\pi)^{-s-1}}{\pi\sqrt{\a\b}}\G(s+1)\sin\left(\frac{\pi s}{2}\right)\left(\frac{\sqrt{\alpha}-\sqrt{\beta}}{\sqrt{\alpha}+\sqrt{\beta}}\right)^{\l}\left(\frac{1}{\sqrt{\alpha}}+\frac{1}{\sqrt{\beta}}\right)^{2s}\nonumber\\
&\quad\times\pFq21{\l-s, -s}{\l+1}{\left(\frac{\sqrt{\alpha}-\sqrt{\beta}}{\sqrt{\alpha}+\sqrt{\beta}}\right)^2}.
\end{align}
Similarly from \eqref{xik1} and \eqref{xik2}, we find that for Re$(\l)>-1$,
\begin{align}\label{xik4}
&\int_0^{\infty}y^{-\frac{s}{2}}g(y)\, dy=0.
\end{align}
Finally, from \eqref{guitra}, \eqref{hsl}, \eqref{f}, \eqref{g}, \eqref{f1}, \eqref{fboth}, \eqref{xik3} and \eqref{xik4}, we deduce that for $-\frac{1}{2}<$ Re$(s)<\frac{1}{2}$, Re$(\l)>0$ and either Re$(s+\l)>0$ or $s=-\l$,
{\allowdisplaybreaks\begin{align*}
&\frac{1}{2\pi}\sum_{n=1}^{\infty}\sigma_s(n)(\mathcal{A}(n)+\mathcal{A}(-n))+\frac{2^{2s}\Gamma(\l)\Gamma(s+1)\sin\left(\frac{\pi s}{2}\right)}{\pi\Gamma(\l+s+1)}\zeta(s+1)\left(\frac{\sqrt{\alpha}-\sqrt{\beta}}{\sqrt{\alpha}+\sqrt{\beta}}\right)^{\l}-\zeta(1-s)h(s, \l)\nonumber\\
&=\frac{2\G(\l+1)(8\pi)^{s}}{\G(s+\l+1)}\sum_{n=1}^{\infty}\sigma_s(n)\bigg\{I_{\l}\left(2\pi\overline{\epsilon}\sqrt{n}(\sqrt{\a}-\sqrt{\b})\right)K_{\l}\left(2\pi\overline{\epsilon}\sqrt{n}(\sqrt{\a}+\sqrt{\b})\right)\nonumber\\
&\quad+I_{\l}\left(2\pi\epsilon\sqrt{n}(\sqrt{\a}-\sqrt{\b})\right)K_{\l}\left(2\pi\epsilon\sqrt{n}(\sqrt{\a}+\sqrt{\b})\right)\bigg\}\nonumber\\
&\quad+\frac{(2\pi)^{-s-1}}{\pi\sqrt{\a\b}}\G(s+1)\sin\left(\frac{\pi s}{2}\right)\zeta(s+1)\left(\frac{\sqrt{\alpha}-\sqrt{\beta}}{\sqrt{\alpha}+\sqrt{\beta}}\right)^{\l}\left(\frac{1}{\sqrt{\alpha}}+\frac{1}{\sqrt{\beta}}\right)^{2s}\nonumber\\
&\quad\times\pFq21{\l-s, -s}{\l+1}{\left(\frac{\sqrt{\alpha}-\sqrt{\beta}}{\sqrt{\alpha}+\sqrt{\beta}}\right)^2}.
\end{align*}}
Now multiply both sides of the above equation by $\frac{\G(s+\l+1)}{2\G(\l+1)(8\pi)^{s}}$, use the asymmetric form of the functional equation of $\zeta(-s)$ \cite[p.~25]{titch}, namely,
\begin{equation}\label{asyfe}
\zeta(-s)=-2^{-s}\pi^{-s-1}\G(1+s)\zeta(1+s)\sin\left(\frac{\pi s}{2}\right)
\end{equation}
to simplify the resulting equation and use \eqref{funa} to arrive at \eqref{anathmbdkzeqn} for $\a>\b>0$, $-\frac{1}{2}<$ Re$(s)<\frac{1}{2}$, Re$(\l)>0$ and either Re$(s+\l)>0$ or $s=-\l$.

Now it can be seen that both sides of \eqref{anathmbdkzeqn} are analytic as long as Re$(\sqrt{\a})>$Re$(\sqrt{\b})>0$, $|\arg\a|<\pi/2$, $|\arg\b|<\pi/2$, Re$(\l)>0$, and $s\in\mathbb{C}$ such that either Re$(s+\l)>0$ or $s=-\l$. To see why this is true even for $s=-1$, note that the power series expansions
\begin{align}\label{pp1}
-\frac{\zeta(-s)}{2\l}\left(\frac{\sqrt{\a}-\sqrt{\b}}{\sqrt{\a}+\sqrt{\b}}\right)^{\l}=\frac{1}{2\l(s+1)}\left(\frac{\sqrt{\a}-\sqrt{\b}}{\sqrt{\a}+\sqrt{\b}}\right)^{\l}-\frac{\g}{2\l}\left(\frac{\sqrt{\a}-\sqrt{\b}}{\sqrt{\a}+\sqrt{\b}}\right)^{\l}+O(|s+1|),
\end{align}
and
\begin{align}\label{pp2}
&\frac{\zeta(-s)}{2^{3s+2}\pi^{s+1}\sqrt{\a\b}}\frac{\G(\l+s+1)}{\G(\l+1)}\left(\frac{\sqrt{\a}-\sqrt{\b}}{\sqrt{\a}+\sqrt{\b}}\right)^{\l}\left(\frac{1}{\sqrt{\a}}+\frac{1}{\sqrt{\b}}\right)^{2s}\pFq21{\l-s, -s}{\l+1}{\left(\frac{\sqrt{\a}-\sqrt{\b}}{\sqrt{\a}+\sqrt{\b}}\right)^2}\nonumber\\
&=-\frac{1}{2\l(s+1)}\left(\frac{\sqrt{\a}-\sqrt{\b}}{\sqrt{\a}+\sqrt{\b}}\right)^{\l}+c_{\a, \b, \l}+O(|s+1|),
\end{align}
where $c_{\a, \b, \l}$ is some constant independent of $s$, clearly show that the principal parts of the left-hand sides of \eqref{pp1} and \eqref{pp2} completely cancel out. 

Hence by analytic continuation, \eqref{anathmbdkzeqn} is valid for Re$(\sqrt{\a})>$Re$(\sqrt{\b})>0$, $|\arg\a|<\pi/2$, $|\arg\b|<\pi/2$, Re$(\l)>0$, and for all complex $s$ such that either Re$(s+\l)>0$ or $s=-\l$.

\section{Proofs of Corollaries of Theorem \ref{anathmbdkz}}\label{cor1}
\begin{proof}[Corollary \textup{\ref{dixmolcor}}][]
First let Re$(\l)>0$. Let $s=-\l$ in Theorem \ref{anathmbdkz}. Then divide both sides by $(\sqrt{\a}-\sqrt{\b})^{\l}$ and let $\a\to\b$. This gives
\begin{align}\label{l0}
&\sum_{n=1}^{\infty}\sigma_{-\l}(n)\lim_{\a\to\b}\frac{1}{(\a-\b)
^{\l}}\bigg\{I_{\l}\left(2\pi\overline{\epsilon}(\sqrt{n\a}-\sqrt{n\b})\right)K_{\l}\left(2\pi\overline{\epsilon}(\sqrt{n\a}+\sqrt{n\b})\right)\nonumber\\
&\qquad\qquad\qquad\qquad\qquad\quad\quad+I_{\l}\left(2\pi\epsilon(\sqrt{n\a}-\sqrt{n\b})\right)K_{\l}\left(2\pi\epsilon(\sqrt{n\a}+\sqrt{n\b})\right)\bigg\}\nonumber\\
&=-\frac{\pi^{\l}\b^{\l/2}\zeta(\l+1)}{4\G(\l+1)}-\frac{\zeta(\l)}{\l 2^{\l+1}\b^{\l/2}}+\frac{\pi^{\l-1}\b^{-1+\l/2}\zeta(\l)}{4\G(\l+1)}\nonumber\\
&\quad+\frac{1}{2^{2-3\l}\pi^{1-\l}\G(\l+1)}\sum_{n=1}^{\infty}\sigma_{-\l}(n)\left\{\frac{2^{\l}\b^{\l/2}(\b+in)^{2\l-1}}{2^{4\l}(\b+in)^{2\l}}+\frac{2^{\l}\b^{\l/2}(\b-in)^{2\l-1}}{2^{4\l}(\b-in)^{2\l}}\right\}.
\end{align}
Let $L:=L(n, \l, \b)$ denote the limit in the above equation. We show that
\begin{equation}\label{l}
L=\frac{(\pi\sqrt{n})^{\l}}{\G(\l+1)}\left(e^{\frac{i\pi\l}{4}}K_{\l}\left(4\pi e^{\frac{i\pi}{4}}\sqrt{n\b}\right)+e^{-\frac{i\pi\l}{4}}K_{\l}\left(4\pi e^{-\frac{i\pi}{4}}\sqrt{n\b}\right)\right).
\end{equation}
To that end, note that Re$(\sqrt{\a})>$Re$(\sqrt{\b})$ implies $-\frac{\pi}{4}<\arg(2\pi\e\sqrt{n}(\sqrt{\a}-\sqrt{\b}))<\frac{3\pi}{4}$. First, consider  $-\frac{\pi}{4}<\arg(2\pi\e\sqrt{n}(\sqrt{\a}-\sqrt{\b}))\leq\frac{\pi}{2}$. Then using \eqref{besseli} and \eqref{sumbesselj}, we see that
\begin{align}\label{l1}
&\lim_{\a\to\b}\frac{I_{\l}\left(2\pi\epsilon(\sqrt{n\a}-\sqrt{n\b})\right)K_{\l}\left(2\pi\epsilon(\sqrt{n\a}+\sqrt{n\b})\right)}{(\sqrt{\a}-\sqrt{\b})
^{\l}}\nonumber\\
&=e^{\frac{-\pi i\l}{2}}K_{\l}(4\pi\e\sqrt{n\b})\lim_{\a\to\b}\frac{J_{\l}\left(2\pi e^{\frac{3\pi i}{4}}\sqrt{n}(\sqrt{\a}-\sqrt{\b})\right)}{(\sqrt{\a}-\sqrt{\b})^{\l}}\nonumber\\
&=e^{\frac{-\pi i\l}{2}}K_{\l}(4\pi\e\sqrt{n\b})\nonumber\\
&\quad\times\lim_{\a\to\b}\frac{1}{(\sqrt{\a}-\sqrt{\b})^{\l}}\left(\frac{e^{\frac{3\pi i\l}{4}}(\pi\sqrt{n})^{\l}(\a-\b)^{\l}}{\G(\l+1)}+\sum_{m=1}^{\infty}\frac{(-1)^m\left(\pi e^{\frac{3\pi i}{4}}(\sqrt{\a}-\sqrt{\b})\right)^{2m+\l}}{m!\G(m+1+\l)}\right)\nonumber\\
&=\frac{(\pi\e\sqrt{n})^{\l}}{\G(\l+1)}K_{\l}(4\pi\e\sqrt{n\b}).
\end{align}
Similarly, it can be shown that we get the same evaluation as above when $\frac{\pi}{2}<\arg(2\pi\e\sqrt{n}(\sqrt{\a}-\sqrt{\b}))<\frac{3\pi}{4}$. 
Also, one can show in the same way that
\begin{equation}\label{l2}
\lim_{\a\to\b}\frac{I_{\l}\left(2\pi\overline{\epsilon}(\sqrt{n\a}-\sqrt{n\b})\right)K_{\l}\left(2\pi\overline{\epsilon}(\sqrt{n\a}+\sqrt{n\b})\right)}{(\sqrt{\a}-\sqrt{\b})^{\l}}=\frac{(\pi\overline{\e}\sqrt{n})^{\l}}{\G(\l+1)}K_{\l}(4\pi\overline{\e}\sqrt{n\b}).
\end{equation}
Thus \eqref{l1} and \eqref{l2} establish \eqref{l}. Now substitute \eqref{l} in \eqref{l0}, multiply both sides of the resulting equation by $\frac{2\G(\l+1)}{\pi^{\l}}$ and simplify so as to obtain \eqref{dixmol} for Re$(\l)>0$. The result extends to Re$(\l)>-1$ by analytic continuation.
\end{proof}
\begin{proof}[Corollary \textup{\ref{dixmolcorgen}}][]
This proof begins by first obtaining a result complementary to \eqref{dixmol}. To be more specific, let $s\neq-\l$ and Re$(s+\l)>0$ in \eqref{anathmbdkzeqn}. Divide both sides of \eqref{anathmbdkzeqn} by by $(\sqrt{\a}-\sqrt{\b})^{\l}$, then let $\a\to\b$ and then simplify so as to deduce
\begin{align*}
&\sum_{n=1}^{\infty} \sigma_{s}(n) n^{\frac{\l}{2}} 
\left( e^{\frac{\pi i \l}{4}} K_{\l}( 4 \pi \e \sqrt{n\b} ) +
 e^{-\frac{\pi i \l}{4}} K_{\l}( 4 \pi \overline{\e} \sqrt{n\b} ) \right)\nonumber\\
&=-\frac{\zeta(-s)\G(\l)}{2^{\l+1}\pi^{\l}\b^{\l/2}}+\frac{\zeta(-s)\G(\l+s+1)}{2(2\pi)^{s+\l+1}\b^{s+\frac{\l}{2}+1}}\nonumber\\
&\quad+\frac{\b^{\frac{\l}{2}}\G(\l+s+1)}{2(2\pi)^{s+\l+1}}\sum_{n=1}^{\infty}\sigma_s(n)\left\{\frac{1}{(\b+in)^{s+\l+1}}+\frac{1}{(\b-in)^{s+\l+1}}\right\}.
\end{align*}
Now \eqref{dixmolcorgeneqn} follows from letting $\l=1/2$, using \eqref{khalf}, and simplifying.
\end{proof}
\begin{proof}[Corollary \textup{\ref{exotic}}][]
Let $s\to-1$ in Theorem \ref{anathmbdkz}. This gives for Re$(\l)>1$ or $\l=1$,
{\allowdisplaybreaks\begin{align}\label{exotics-1}
&\sum_{n=1}^{\infty}\sigma_{-1}(n)\bigg\{I_{\l}\left(2\pi\overline{\epsilon}(\sqrt{n\a}-\sqrt{n\b})\right)K_{\l}\left(2\pi\overline{\epsilon}(\sqrt{n\a}+\sqrt{n\b})\right)\nonumber\\
&\qquad\qquad+I_{\l}\left(2\pi\epsilon(\sqrt{n\a}-\sqrt{n\b})\right)K_{\l}\left(2\pi\epsilon(\sqrt{n\a}+\sqrt{n\b})\right)\bigg\}\nonumber\\
&=-\frac{2\pi^3}{3\l}h(-1, \l)+\left(\frac{\sqrt{\a}-\sqrt{\b}}{\sqrt{\a}+\sqrt{\b}}\right)^{\l}L_{1}(\a, \b, \l)\nonumber\\
&\quad+\frac{2}{\l}\sum_{n=1}^{\infty}\sigma_{-1}(n)\left(\mathcal{A}\left(-1, \l, \sqrt{\a}, \sqrt{\b}, n\right)+\mathcal{A}\left(-1, \l, \sqrt{\a}, \sqrt{\b}, -n\right)\right),
\end{align}}
where
\begin{align}\label{lalim}
L_{1}(\a, \b, \l)&:=\lim_{s\to-1}\bigg\{-\frac{\zeta(-s)}{2\l}+\frac{\zeta(-s)}{2^{3s+2}\pi^{s+1}\sqrt{\a\b}}\frac{\G(\l+s+1)}{\G(\l+1)}\nonumber\\
&\quad\times\left(\frac{1}{\sqrt{\a}}+\frac{1}{\sqrt{\b}}\right)^{2s}\pFq21{\l-s, -s}{\l+1}{\left(\frac{\sqrt{\a}-\sqrt{\b}}{\sqrt{\a}+\sqrt{\b}}\right)^2}\bigg\}.
\end{align}
The above limit is now evaluated. Note that the Laurent series expansion of $\zeta(s)$ around $s=1$ gives
\begin{equation}\label{laurent}
\zeta(s)=\frac{1}{s-1}+\g+O(|s-1|).
\end{equation}
Hence
\begin{align}\label{one}
-\frac{\zeta(-s)}{2\l}=\frac{1}{2\l(s+1)}-\frac{\g}{2\l}+O(|s+1|).
\end{align}
as $s\to-1$. Also,
\begin{align}
b^{s}&=\frac{1}{b}+\frac{\log b}{b}(s+1)+O(|s+1|^2),\label{bs}\\
\G(\l+s+1)&=\G(\l)+\G'(\l)(s+1)+O(|s+1|^2),\label{gls}
\end{align} 	
and
\begin{align}\label{hype}
&\pFq21{\l-s, -s}{\l+1}{\left(\frac{\sqrt{\a}-\sqrt{\b}}{\sqrt{\a}+\sqrt{\b}}\right)^2}\nonumber\\
&=\frac{1}{1-\left(\frac{\sqrt{\a}-\sqrt{\b}}{\sqrt{\a}+\sqrt{\b}}\right)^2}+(s+1)\lim_{s\to-1}\left[\frac{d}{ds}\pFq21{\l-s, -s}{\l+1}{\left(\frac{\sqrt{\a}-\sqrt{\b}}{\sqrt{\a}+\sqrt{\b}}\right)^2}\right]+O(|s+1|^2).
\end{align}
Next,
\begin{align}
&\lim_{s\to-1}\left[\frac{d}{ds}\pFq21{\l-s, -s}{\l+1}{\left(\frac{\sqrt{\a}-\sqrt{\b}}{\sqrt{\a}+\sqrt{\b}}\right)^2}\right]\nonumber\\
&=\sum_{k=1}^{\infty}\frac{1}{k!(1+\l)_k}\left(\frac{\sqrt{\a}-\sqrt{\b}}{\sqrt{\a}+\sqrt{\b}}\right)^{2k}\lim_{s\to-1}\left[(-s)_k\frac{d}{ds}(\l-s)_k+(\l-s)_k\frac{d}{ds}(-s)_k\right]
\end{align}
Using the fact
\begin{align*}
\frac{d}{d\xi}(\xi)_k=\frac{d}{d\xi}\frac{\G(\xi+k)}{\G(\xi)}=\frac{\G(\xi+k)}{\G(\xi)}\left(\psi(\xi+k)-\psi(\xi)\right),
\end{align*}
we find, after simplification, that
\begin{align}\label{derev}
&\lim_{s\to-1}\left[\frac{d}{ds}\pFq21{\l-s, -s}{\l+1}{\left(\frac{\sqrt{\a}-\sqrt{\b}}{\sqrt{\a}+\sqrt{\b}}\right)^2}\right]\nonumber\\
&=-\frac{(\sqrt{\a}+\sqrt{\b})^2}{4\sqrt{\a\b}}\left[\sum_{k=0}^{\infty}\frac{1}{1+\l+k}\left(\frac{\sqrt{\a}-\sqrt{\b}}{\sqrt{\a}+\sqrt{\b}}\right)^{2k+2}-\log\left(\frac{4\sqrt{\a\b}}{(\sqrt{\a}+\sqrt{\b})^2}\right)\right].
\end{align}
Now substitute \eqref{derev} in \eqref{hype} and use the resultant along with \eqref{bs} with $b=\frac{1}{8\pi}\left(\frac{1}{\sqrt{\a}}+\frac{1}{\sqrt{\b}}\right)^2$, and \eqref{gls} to obtain
\begin{align}\label{two}
&\frac{\zeta(-s)}{2^{3s+2}\pi^{s+1}\sqrt{\a\b}}\frac{\G(\l+s+1)}{\G(\l+1)}\left(\frac{1}{\sqrt{\a}}+\frac{1}{\sqrt{\b}}\right)^{2s}\pFq21{\l-s, -s}{\l+1}{\left(\frac{\sqrt{\a}-\sqrt{\b}}{\sqrt{\a}+\sqrt{\b}}\right)^2}\nonumber\\
&=\frac{1}{4\pi\sqrt{\a\b}\G(\l+1)}\left(\frac{1}{8\pi}\left(\frac{1}{\sqrt{\a}}+\frac{1}{\sqrt{\b}}\right)^2\right)^s\G(\l+s+1)\zeta(-s)\pFq21{\l-s, -s}{\l+1}{\left(\frac{\sqrt{\a}-\sqrt{\b}}{\sqrt{\a}+\sqrt{\b}}\right)^2}\nonumber\\
&=-\frac{1}{2\l(s+1)}+\frac{\g}{2\l}+\frac{1}{2\l}\log(2\pi\sqrt{\a\b})-\frac{1}{2\l}\psi(\l)+\frac{1}{2\l}\sum_{k=0}^{\infty}\frac{1}{1+\l+k}\left(\frac{\sqrt{\a}-\sqrt{\b}}{\sqrt{\a}+\sqrt{\b}}\right)^{2k+2}\nonumber\\
&\quad+O(|s+1|)
\end{align}
as $s\to-1$. Now add the corresponding sides of \eqref{one} and \eqref{two} so as to deduce from \eqref{lalim} that
\begin{align*}
L_{1}(\a, \b, \l)=\frac{1}{2\l}\left\{\sum_{k=0}^{\infty}\frac{1}{1+\l+k}\left(\frac{\sqrt{\a}-\sqrt{\b}}{\sqrt{\a}+\sqrt{\b}}\right)^{2k+2}+\log(2\pi\sqrt{\a\b})-\psi(\l)\right\}.
\end{align*}
Finally let $\l=1$ in \eqref{exotics-1}, simplify the above limit, use the facts $h(-1,1)=(\a-\b)/32$, where $h(s,\l)$ is defined in \eqref{hsl}, $\sum_{k=0}^{\infty}\xi^k/(k+2)=-1-\frac{\log(1-\xi)}{\xi}$ and $\mathcal{A}\left(-1, 1, \sqrt{\a}, \sqrt{\b}, n\right)=\frac{1}{4}\left(\frac{\sqrt{\a+in}-\sqrt{\b+in}}{\sqrt{\a+in}+\sqrt{\b+in}}\right)$
to arrive at \eqref{exoticeqn} after simplification.
\end{proof}
\section{Analytic continuation of Theorem \ref{anathmbdkz}}\label{ac}
Here we analytically continue Theorem \ref{anathmbdkz} to complex values of $\l$ such that Re$(\l)>-1$. The special case $\l=-1/2$ of the resulting identity is then shown to have a form surprisingly similar to Ramanujan's \eqref{ramin}.
\begin{theorem}\label{anathmbdkzac}
Let $\mathcal{A}(s, \l, z, w, t)$ be defined in \eqref{funa}. Let $|\arg \a|<\pi/2, |\arg \b|<\pi/2$ with \textup{Re}$(\sqrt{\a})>$ \textup{Re}$(\sqrt{\b})$. Let \textup{Re}$(\l)>-1$. 
Then for all complex $s$ such that either \textup{Re}$(s+\l)>0$ or $s=-\l$,
\begin{align}\label{anathmbdkzeqnac}
&\sum_{n=1}^{\infty}\sigma_s(n)\bigg\{I_{\l}\left(2\pi\overline{\epsilon}(\sqrt{n\a}-\sqrt{n\b})\right)K_{\l}\left(2\pi\overline{\epsilon}(\sqrt{n\a}+\sqrt{n\b})\right)\nonumber\\
&\qquad\qquad+I_{\l}\left(2\pi\epsilon(\sqrt{n\a}-\sqrt{n\b})\right)K_{\l}\left(2\pi\epsilon(\sqrt{n\a}+\sqrt{n\b})\right)\bigg\}\nonumber\\
&=-\frac{\zeta(-s)}{2\l}\left(\frac{\sqrt{\a}-\sqrt{\b}}{\sqrt{\a}+\sqrt{\b}}\right)^{\l}+\frac{\zeta(-s)}{2^{3s+2}\pi^{s+1}\sqrt{\a\b}}\frac{\G(\l+s+1)}{\G(\l+1)}\nonumber\\
&\quad\times\left(\frac{\sqrt{\a}-\sqrt{\b}}{\sqrt{\a}+\sqrt{\b}}\right)^{\l}\left(\frac{1}{\sqrt{\a}}+\frac{1}{\sqrt{\b}}\right)^{2s}\pFq21{\l-s, -s}{\l+1}{\left(\frac{\sqrt{\a}-\sqrt{\b}}{\sqrt{\a}+\sqrt{\b}}\right)^2}\nonumber\\
&\quad+\frac{\pi^{\l}(\a-\b)^{\l}\zeta(\l+1)\zeta(-s-\l)}{2^{\l+1}\G(\l+1)}+\frac{\G(\l+s+1)}{2^{3s+2}\pi^{s+1}\G(\l+1)}\sum_{n=1}^{\infty}\sigma_s(n)\bigg[\mathcal{A}\left(s, \l, \sqrt{\a}, \sqrt{\b}, n\right)\nonumber\\
&\quad+\mathcal{A}\left(s, \l, \sqrt{\a}, \sqrt{\b}, -n\right)+\frac{2^{2s-2\l+1}(\a-\b)^{\l}}{n^{\l+s+1}}\sin\left(\tfrac{\pi}{2}(\l+s)\right)\bigg].
\end{align}
\end{theorem}
\begin{proof}
First let Re$(\l)>0$ and Re$(s+\l)>0$. Rewrite \eqref{anathmbdkzeqn} in the form
\begin{align}\label{anathmbdkzeqno}
&\sum_{n=1}^{\infty}\sigma_s(n)\bigg\{I_{\l}\left(2\pi\overline{\epsilon}(\sqrt{n\a}-\sqrt{n\b})\right)K_{\l}\left(2\pi\overline{\epsilon}(\sqrt{n\a}+\sqrt{n\b})\right)\nonumber\\
&\qquad\qquad+I_{\l}\left(2\pi\epsilon(\sqrt{n\a}-\sqrt{n\b})\right)K_{\l}\left(2\pi\epsilon(\sqrt{n\a}+\sqrt{n\b})\right)\bigg\}\nonumber\\
&=-\frac{\zeta(-s)}{2\l}\left(\frac{\sqrt{\a}-\sqrt{\b}}{\sqrt{\a}+\sqrt{\b}}\right)^{\l}+\frac{\zeta(-s)}{2^{3s+2}\pi^{s+1}\sqrt{\a\b}}\frac{\G(\l+s+1)}{\G(\l+1)}\left(\frac{\sqrt{\a}-\sqrt{\b}}{\sqrt{\a}+\sqrt{\b}}\right)^{\l}\nonumber\\
&\quad\times\left(\frac{1}{\sqrt{\a}}+\frac{1}{\sqrt{\b}}\right)^{2s}\pFq21{\l-s, -s}{\l+1}{\left(\frac{\sqrt{\a}-\sqrt{\b}}{\sqrt{\a}+\sqrt{\b}}\right)^2}+\frac{\G(\l+s+1)}{2^{3s+2}\pi^{s+1}\G(\l+1)}\nonumber\\
&\quad\times\sum_{n=1}^{\infty}\sigma_s(n)\bigg[\mathcal{A}\left(s, \l, \sqrt{\a}, \sqrt{\b}, n\right)+\mathcal{A}\left(s, \l, \sqrt{\a}, \sqrt{\b}, -n\right)+\tfrac{2^{2s-2\l+1}(\a-\b)^{\l}}{n^{\l+s+1}}\sin\left(\tfrac{\pi}{2}(\l+s)\right)\bigg]\nonumber\\
&\quad-\frac{2^{-s-2\l-1}(\a-\b)^{\l}\G(\l+s+1)}{\pi^{s+1}\G(\l+1)}\zeta(\l+s+1)\zeta(\l+1)\sin\left(\frac{\pi}{2}(\l+s)\right),
\end{align}
where in the last step, we used \cite[p.~8, Equation (1.3.1)]{titch}
\begin{equation*}
\sum_{n=1}^{\infty}\frac{\sigma_s(n)}{n^{\l+s+1}}=\zeta(\l+s+1)\zeta(\l+1),
\end{equation*}
which is valid for Re$(\l)>0$ and Re$(s+\l)>0$.

Now it can be easily seen that
\begin{align}\label{ae}
\mathcal{A}\left(s, \l, \sqrt{\a}, \sqrt{\b}, n\right)+\mathcal{A}\left(s, \l, \sqrt{\a}, \sqrt{\b}, -n\right)&=-\frac{2^{2s-2\l+1}(\a-\b)^{\l}}{n^{\l+s+1}}\sin\left(\frac{\pi}{2}(\l+s)\right)\nonumber\\
&\quad+O_{\a, \b, \l, s}\left(n^{-\textup{Re}(\l)-\textup{Re}(s)-2}\right),
\end{align}
which implies that the series on the right-hand side of \eqref{anathmbdkzeqno} is analytic for Re$(\l)>-1$. Also
\begin{equation*}
-\frac{2^{-s-2\l-1}(\a-\b)^{\l}\G(\l+s+1)}{\pi^{s+1}\G(\l+1)}\zeta(\l+s+1)\zeta(\l+1)\sin\left(\frac{\pi}{2}(\l+s)\right)
\end{equation*}
is analytic for Re$(s+\l)>0$ and Re$(\l)>-1$ except for a simple pole at $\l=0$. However, the contribution of this pole is nullified by that of the simple pole of $-\frac{\zeta(-s)}{2\l}\left(\frac{\sqrt{\a}-\sqrt{\b}}{\sqrt{\a}+\sqrt{\b}}\right)^{\l}$ at $\l=0$. Thus by analytic continuation, we see that \eqref{anathmbdkzeqno} holds for Re$(\l)>-1$. Now use \eqref{asyfe} with $s$ replaced by $s+\l$ and simplify the right-hand side of \eqref{anathmbdkzeqno}. This proves \eqref{anathmbdkzeqnac} for Re$(\l)>-1$ when Re$(s+\l)>0$. 

Now if $s=-\l$, note that the leading term in the asymptotic expansion in \eqref{ae} vanishes and so instead of \eqref{anathmbdkzeqno}, we work with \eqref{anathmbdkzeqn} only. As in the previous case, it can be shown by analytic continuation that for $s=-\l$, \eqref{anathmbdkzeqn} is actually valid for Re$(\l)>-1$. It is now pleasing to note that the term containing $h(s, \l)$ in \eqref{anathmbdkzeqn} for $s=-\l$ is exactly the same as the term 
\begin{equation*}
\frac{\pi^{\l}(\a-\b)^{\l}\zeta(\l+1)\zeta(-s-\l)}{2^{\l+1}\G(\l+1)}
\end{equation*}
occurring in \eqref{anathmbdkzeqnac}, as can be seen from the fact that $\zeta(0)=-1/2$. Hence we conclude that \eqref{anathmbdkzeqnac} holds for Re$(\l)>-1$ even when $s=-\l$.

 %

\end{proof}

A few interesting corollaries of the above theorem are now given.
\begin{corollary}\label{ramanujan-type}
Let $\textup{Re}(\b)>0$. For \textup{Re}$(s)>-1/2$,
\begin{align}\label{ramanujan-typecor}
&\G\left(s+\frac{1}{2}\right)\bigg\{\frac{\zeta(-s)}{2\b^{s+\frac{1}{2}}}+\sum_{n=1}^{\infty}\frac{\sigma_s(n)}{2}\bigg[(\b-in)^{-s-\frac{1}{2}}+(\b+in)^{-s-\frac{1}{2}}\nonumber\\
&\qquad\qquad\quad+\frac{2}{n^{s+\frac{1}{2}}}\sin\bigg(\frac{\pi}{2}\bigg(s-\frac{1}{2}\bigg)\bigg)\bigg]\bigg\}\nonumber\\
&=(2\pi)^s\bigg\{-2\pi\sqrt{\pi \b}\zeta(-s)-\sqrt{\frac{\pi}{2}}\zeta\left(\frac{1}{2}\right)\zeta\left(\frac{1}{2}-s\right)\nonumber\\
&\qquad\qquad+\sqrt{\pi}\sum_{n=1}^{\infty}\frac{\sigma_s(n)}{\sqrt{n}}e^{-2\pi\sqrt{2n\b}}
\sin\left(\frac{\pi}{4}-2\pi\sqrt{2n\b}\right)\bigg\}.
\end{align}
\end{corollary}
\begin{proof}
Let $\l=-1/2$ in Theorem \ref{anathmbdkzac}. Using \eqref{khalf} and \cite[p.~925, formulas \textbf{8.467}, \textbf{8.469.3}]{grn}
\begin{equation*}
I_{-\frac{1}{2}}(\xi)=\sqrt{\frac{2}{\pi\xi}}\cosh(\xi),
\end{equation*}
we see that the left-hand side of \eqref{anathmbdkzeqnac} simplifies to
\begin{align}\label{lhs}
&\frac{1}{2\pi\sqrt{\a-\b}}\sum_{n=1}^{\infty}\frac{\sigma_s(n)}{\sqrt{n}}\bigg\{\e \exp{\left(-2\pi\overline{\e}\sqrt{n}(\sqrt{\a}+\sqrt{\b})\right)}\cosh\left(2\pi\overline{\e}\sqrt{n}(\sqrt{\a}-\sqrt{\b})\right)\nonumber\\
&\qquad\qquad\qquad\qquad\quad+\overline{\e} \exp{\left(-2\pi\e\sqrt{n}(\sqrt{\a}+\sqrt{\b})\right)}\cosh\left(2\pi\e\sqrt{n}(\sqrt{\a}-\sqrt{\b})\right)\bigg\},
\end{align}
where as the right-hand side, upon using the identity \cite[p.~389, Formula \textbf{7.3.1.106}]{pbm}
\begin{equation*}
\pFq21{-\frac{1}{2}-s, -s}{\frac{1}{2}}{\xi}=\frac{1}{2}\left\{\left(1-\sqrt{\xi}\right)^{2s+1}+\left(1+\sqrt{\xi}\right)^{2s+1}\right\},
\end{equation*}
results in 
\begin{align}\label{rhs}
&\frac{\zeta(-s)(\sqrt{\a}+\sqrt{\b})}{\sqrt{\a-\b}}+\frac{\G\left(s+\frac{1}{2}\right)\zeta(-s)}{2^{s+2}\pi^{s+\frac{3}{2}}\sqrt{\a-\b}}\left(\frac{1}{\a^{s+\frac{1}{2}}}+\frac{1}{\b^{s+\frac{1}{2}}}\right)+\frac{\zeta\left(\frac{1}{2}\right)\zeta\left(\frac{1}{2}-s\right)}{\pi\sqrt{2}\sqrt{\a-\b}}\nonumber\\
&+\frac{\G\left(s+\frac{1}{2}\right)}{2^{s+2}\pi^{s+\frac{3}{2}}\sqrt{\a-\b}}\sum_{n=1}^{\infty}\sigma_{s}(n)\bigg[\frac{1}{(\a+in)^{s+\frac{1}{2}}}+\frac{1}{(\b+in)^{s+\frac{1}{2}}}\nonumber\\
&\quad+\frac{1}{(\a-in)^{s+\frac{1}{2}}}+\frac{1}{(\b-in)^{s+\frac{1}{2}}}+\frac{4}{n^{s+\frac{1}{2}}}\sin\left(\frac{\pi}{2}\left(s-\frac{1}{2}\right)\right)\bigg].
\end{align}
Now multiply the expressions in \eqref{lhs} and \eqref{rhs} by $\pi\sqrt{\a-\b}$ and let $\a\to\b$ in the resulting equality so as to get upon simplification
\begin{align}\label{rem3}
&\sum_{n=1}^{\infty}\frac{\sigma_s(n)}{\sqrt{n}}e^{-2\pi\sqrt{2n\b}}\sin\left(\frac{\pi}{4}-2\pi\sqrt{2n\b}\right)\nonumber\\
&=2\pi\sqrt{\b}\zeta(-s)+\frac{\G\left(s+\frac{1}{2}\right)\zeta(-s)}{\sqrt{2}(2\pi\b)^{s+\frac{1}{2}}}+\frac{1}{\sqrt{2}}\zeta\left(\frac{1}{2}\right)\zeta\left(\frac{1}{2}-s\right)\nonumber\\
&\quad+\frac{\G\left(s+\frac{1}{2}\right)}{\sqrt{2}(2\pi)^{s+\frac{1}{2}}}\sum_{n=1}^{\infty}\sigma_s(n)\left[\frac{1}{(\b+in)^{s+\frac{1}{2}}}+\frac{1}{(\b-in)^{s+\frac{1}{2}}}+\frac{2}{n^{s+\frac{1}{2}}}\sin\left(\frac{\pi}{2}\left(s-\frac{1}{2}\right)\right)\right].
\end{align}
Finally multiply both sides by $2^s\pi^{s+\frac{1}{2}}$ and rearrange so as to arrive at \eqref{ramanujan-typecor} for Re$(s)>1/2$ or $s=1/2$. By analytic continuation, the result holds for Re$(s)>-1/2$.
\end{proof}
\begin{remark}
Corollary \ref{ramanujan-type} bears striking resemblance to Ramanujan's incorrect identity \eqref{ramin}. Even though the summands of the corresponding infinite series in \eqref{ramanujan-typecor} and \eqref{ramin} have wrong signs, the infinite series in \eqref{ramanujan-typecor} remarkably have the same forms compared to those in \eqref{ramin}. Same is the case with two of the residual terms in \eqref{ramanujan-typecor}, which, except for the absence of the factor $\tan\left(\frac{1}{2}\pi s\right)$ in them, are exactly the same as the corresponding ones in \eqref{ramin}. We thus claim that our Corollary \ref{ramanujan-type} can be considered as the corrected version of Ramanujan's identity \eqref{ramin}.
\end{remark}
\begin{remark}
It can be shown without much effort that Theorem \textup{1.3} from \cite{bdrz} is actually equivalent to Corollary \ref{ramanujan-type}. However, because the former was left in an unsimplified form in general in \cite{bdrz}, its closeness to Ramanujan's incorrect identity \eqref{ramin} that we have demonstrated above got unnoticed. The form was simplified only for $s=2m+1/2, m\in\mathbb{N}\cup\{0\}$, for example, see \cite[Theorem 5.2]{bdrz}. In light of this, we would like to emphasize here that Theorem \textup{1.3} from \cite{bdrz} is but a special case of our Theorem \textup{\ref{anathmbdkzac}}.
\end{remark}
\begin{remark}
The case $\l=\frac{1}{2}\neq-s$ of Theorem \ref{anathmbdkz} results in an identity which can be obtained by merely replacing $\b$ by $\a$ in \eqref{rem3} and then substracting \eqref{rem3} as it is from the resulting former identity. 
\end{remark}

\noindent
Corollary \ref{ramanujan-type}, in turn, gives the following new formula for $\zeta^{2}(1/2)$.
\begin{corollary}\label{ramanujan-type0}
For \textup{Re}$(\b)>0$,
\begin{align*}
&\sum_{n=1}^{\infty}\frac{d(n)}{\sqrt{n}}e^{-2\pi\sqrt{2n\b}}\sin\left(\frac{\pi}{4}-2\pi\sqrt{2n\b}\right)\nonumber\\
&=\frac{1}{\sqrt{2}}\zeta^{2}\left(\frac{1}{2}\right)-\pi\sqrt{\b}-\frac{1}{4\sqrt{\b}}+\frac{1}{2}\sum_{n=1}^{\infty}d(n)\left(\frac{1}{\sqrt{\b+in}}+\frac{1}{\sqrt{\b-in}}-\frac{\sqrt{2}}{\sqrt{n}}\right).
\end{align*}
\end{corollary}
\begin{proof}
Set $s=0$ in Corollary \ref{ramanujan-type}, or equivalently, in \eqref{rem3}, and simplify.
\end{proof}
Another special case of Theorem \ref{anathmbdkzac} is
\begin{corollary}\label{7.4}
Let $|\arg \a|<\pi/2, |\arg \b|<\pi/2$ with \textup{Re}$(\sqrt{\a})>$\textup{Re}$(\sqrt{\b})$. Then
\begin{align}\label{anathmbdkzeqnac00}
&\sum_{n=1}^{\infty}d(n)\bigg\{I_{0}\left(2\pi\overline{\epsilon}(\sqrt{n\a}-\sqrt{n\b})\right)K_{0}\left(2\pi\overline{\epsilon}(\sqrt{n\a}+\sqrt{n\b})\right)\nonumber\\
&\qquad\qquad+I_{0}\left(2\pi\epsilon(\sqrt{n\a}-\sqrt{n\b})\right)K_{0}\left(2\pi\epsilon(\sqrt{n\a}+\sqrt{n\b})\right)\bigg\}\nonumber\\
&=-\frac{\g}{2}+\frac{1}{2}\log\left(\tfrac{2}{\sqrt{\a}+\sqrt{\b}}\right)-\frac{1}{8\pi\sqrt{\a\b}}+\frac{1}{4\pi}\sum_{n=1}^{\infty}d(n)\left\{\frac{1}{\sqrt{\a+in}\sqrt{\b+in}}+\frac{1}{\sqrt{\a-in}\sqrt{\b-in}}\right\}.
\end{align}
\end{corollary}
\begin{proof}
Let $\l=0$ in Theorem \ref{anathmbdkzeqnac}. This gives for Re$(s)>0$ or $s=0$,
{\allowdisplaybreaks\begin{align}\label{anathmbdkzeqnac001}
&\sum_{n=1}^{\infty}\sigma_s(n)\bigg\{I_{0}\left(2\pi\overline{\epsilon}(\sqrt{n\a}-\sqrt{n\b})\right)K_{0}\left(2\pi\overline{\epsilon}(\sqrt{n\a}+\sqrt{n\b})\right)\nonumber\\
&\qquad\qquad+I_{0}\left(2\pi\epsilon(\sqrt{n\a}-\sqrt{n\b})\right)K_{0}\left(2\pi\epsilon(\sqrt{n\a}+\sqrt{n\b})\right)\bigg\}\nonumber\\
&=L_2(s, \a, \b)+\frac{\zeta(-s)\G(s+1)}{2^{3s+2}\pi^{s+1}\sqrt{\a\b}}\left(\frac{1}{\sqrt{\a}}+\frac{1}{\sqrt{\b}}\right)^{2s}\pFq21{-s, -s}{1}{\left(\frac{\sqrt{\a}-\sqrt{\b}}{\sqrt{\a}+\sqrt{\b}}\right)^2}\nonumber\\
&\quad+\frac{\G(s+1)}{2^{3s+2}\pi^{s+1}}\sum_{n=1}^{\infty}\sigma_s(n)\left\{\mathcal{A}\left(s,0,\sqrt{\a}, \sqrt{\b}, n\right)+\mathcal{A}\left(s,0,\sqrt{\a}, \sqrt{\b}, -n\right)+\frac{2^{2s+1}}{n^{s+1}}\sin\left(\frac{\pi s}{2}\right)\right\},
\end{align}}
where
\begin{align}\label{ell1sab}
L_2(s, \a, \b):=\lim_{\l\to 0}\left(-\frac{\zeta(-s)}{2\l}\left(\frac{\sqrt{\a}-\sqrt{\b}}{\sqrt{\a}+\sqrt{\b}}\right)^{\l}+\frac{\pi^{\l}(\a-\b)^{\l}\zeta(\l+1)\zeta(-s-\l)}{2^{\l+1}\G(\l+1)}\right).
\end{align}
We only need to evaluate the above limit. Now
\begin{align}\label{l11}
-\frac{\zeta(-s)}{2\l}\left(\frac{\sqrt{\a}-\sqrt{\b}}{\sqrt{\a}+\sqrt{\b}}\right)^{\l}=-\frac{\zeta(-s)}{2\l}-\frac{1}{2}\zeta(-s)\log\left(\frac{\sqrt{\a}-\sqrt{\b}}{\sqrt{\a}+\sqrt{\b}}\right)+O_{\a, \b, s}(|\l|)
\end{align}
as $\l\to0$. The power series expansion of $1/\G(\xi)$ about $\xi=0$ \cite[Equation (22)]{wrench} implies
\begin{equation}\label{wrench}
\frac{1}{\G(\xi)}=\xi+\g\xi^2+O(|\xi|^3).
\end{equation}
Now \eqref{laurent}, \eqref{wrench}, along with the standard Taylor expansions of $\left(\frac{\pi}{2}(\a-\b)\right)^{\l}$ and $\zeta(-s-\l)$ about $\l=0$, together give
\begin{align}\label{l12}
\frac{\pi^{\l}(\a-\b)^{\l}\zeta(\l+1)\zeta(-s-\l)}{2^{\l+1}\G(\l+1)}&=\frac{1}{2}\left\{1+\l\log\left(\tfrac{\pi}{2}(\a-\b)\right)+O_{\a,\b}(|\l|^2)\right\}\left\{\tfrac{1}{\l}+\g+O(|\l|)\right\}\nonumber\\
&\quad\times\left\{\zeta(-s)-\l\zeta'(-s)+O_{s}(|\l|^2)\right\}\left\{1+\g\l+O(|\l|^2)\right\}\nonumber\\
&=\frac{\zeta(-s)}{2\l}+\frac{1}{2}\left(2\g+\log\left(\frac{\pi}{2}(\a-\b)\right)\right)\zeta(-s)-\frac{1}{2}\zeta'(-s)\nonumber\\
&\quad+O_{s, \a, \b}(|\l|).
\end{align}
From \eqref{l11} and \eqref{l12},
\begin{align}\label{l13}
L_2(s, \a, \b)=\left(\g+\frac{1}{2}\log\left(\frac{\pi}{2}\left(\sqrt{\a}+\sqrt{\b}\right)^2\right)\right)\zeta(-s)-\frac{1}{2}\zeta'(-s).
\end{align}
Finally substitute \eqref{l13} in \eqref{anathmbdkzeqnac001}, then let $s=0$ and simplify to derive \eqref{anathmbdkzeqnac00}.
\end{proof}
\begin{remark}
If we let $\a\to\b$ in \eqref{anathmbdkzeqnac00}, use the fact that $I_{0}(0)=1$ and simplify, we obtain \eqref{koshvor}. Thus, \eqref{anathmbdkzeqnac00} is a new generalization of \eqref{koshvor}, different from \eqref{dixmol}.
\end{remark}
The corollary of Theorem \ref{anathmbdkzac} stated below gives a transformation for the series which is almost the same as the left-hand side of \eqref{dixmol} except that $\sigma_{-\l}(n)$ is replaced by $\sigma_{\l}(n)$.
\begin{corollary}\label{soninetrans}
Let \textup{Re}$(\beta)>0$. For \textup{Re}$(\l)>0$ or $\l=0$,
\begin{align*}
&2 \sum_{n=1}^{\infty} \sigma_{\l}(n) n^{\frac{\l}{2}} 
\left( e^{\frac{\pi i \l}{4}} K_{\l}( 4 \pi e^{\frac{\pi i}{4}} \sqrt{n\b} ) +
 e^{-\frac{\pi i \l}{4}} K_{\l}( 4 \pi e^{-\frac{\pi i}{4}} \sqrt{n\b} ) \right)\nonumber\\
&=-\frac{\G(\l)\zeta(-\l)}{(2\pi\sqrt{\b})^{\l}}+\frac{\G(2\l+1)\zeta(-\l)}{(2\pi)^{2\l+1}\b^{\frac{3\l}{2}+1}}+\b^{\frac{\l}{2}}\zeta(\l+1)\zeta(-2\l)\nonumber\\
&\quad+\frac{\b^{\frac{\l}{2}}\G(2\l+1)}{(2\pi)^{2\l+1}}\sum_{n=1}^{\infty}\sigma_{\l}(n)\left\{\frac{1}{(\b+in)^{2\l+1}}+\frac{1}{(\b-in)^{2\l+1}}+\frac{2\sin(\pi\l)}{n^{2\l+1}}\right\}.
\end{align*}
\end{corollary}
\begin{proof}
Let $s=\l$ in Theorem \ref{anathmbdkzac}, divide both sides of the resulting identity by $(\sqrt{\a}-\sqrt{\b})^{\l}$, then let $\a\to\b$ and then simplify.
\end{proof}
\begin{remark}
Letting $\l\to0$ in the above corollary and observing that
\begin{equation*}
\lim_{\l\to0}\left(-\frac{\G(\l)\zeta(-\l)}{(2\pi\sqrt{\b})^{\l}}+\b^{\frac{\l}{2}}\zeta(\l+1)\zeta(-2\l)\right)=-\frac{1}{2}\g-\frac{1}{4}\log\b,
\end{equation*}
which can be proved along similar lines as \eqref{ell1sab}, we arrive at \eqref{koshvor}. Hence we obtain a yet another new generalization of \eqref{koshvor}, different from \eqref{dixmol} and \eqref{anathmbdkzeqnac00}.
\end{remark}
\begin{remark}
If we let $\l=1/2$ in Corollary \textup{\ref{soninetrans}}, we obtain the special case $s=1/2$ of Corollary \textup{\ref{dixmolcorgen}}.
\end{remark}

\section{Concluding Remarks}\label{cr}
The integral \eqref{fock} has its origins in the work of Fock and Bursian on electromagnetism of alternating current in a circuit with two groundings. Number-theoretic applications of its generalization, due to Koshliakov, namely \eqref{koshfock}, were given in \cite{bdrz}. In particular, integrals with kernel $J_{0}(\rho t)$, or more generally $J_{s}(\rho t)$, can be used in the Vorono\"{\dotlessi} summation formula \eqref{guirknsuminfallk} for $r_k(n)$ to obtain results which are of importance in the study of Gauss circle problem. For example, one such result is \cite[Theorem 1.6]{bdkz}, which gives as a special case, a well-known result of Hardy \cite[Equation (2.12)]{hardyqjpam1915} which he used to obtain his famous omega bound for $\sum_{n\leq x}r_2(n)$, see \eqref{harob}. 

Our work was motivated by the fact that since the Vorono\"{\dotlessi} summation formula for $\sigma_s(n)$, which is instrumental in the study of the generalized Dirichlet divisor problem, involves an integral transform consisting of the kernel $F_{s}(\rho t)$ in \eqref{kk}, it may be of interest to seek an analogue of \eqref{koshfock} with $J_{s}(\rho t)$ replaced by this kernel. We highlight here that there aren't many explicit integral valuations of integrals containing $F_{s}(\rho t)$ in its integrand, see for example, \cite[p.~161]{dixfer3}. The aforementioned analogue is found in \eqref{tfockaeqn} of this paper. It would be very interesting and important to see if \eqref{tfockaeqn} has applications in physics and geology, just like its counterpart in \eqref{koshfock}, or more specifically, in \eqref{fock}.

Equation \eqref{tfockaeqn} was also successfully used in this paper to obtain a general transformation \eqref{anathmbdkzeqn}. As shown in this paper, \eqref{anathmbdkzeqn} and its analytic continuation in \eqref{anathmbdkzeqnac} are rich sources of important number-theoretic identities such as Corollary \ref{dixmolcorgen} and Corollary \ref{ramanujan-type}. Corollary \ref{exotic} is remarkable and is the only instance among identities derived in this paper where the hypergeometric function in \eqref{funa} reduces to a non-trivial closed form. 
Of course, one can generalize our results by replacing $\sigma_s(n)$ by coefficients of those Dirichlet series which satisfy functional equations with two gamma factors.

An important special case of Theorem \ref{anathmbdkzac} is Corollary \ref{ramanujan-type}. This result comes remarkably close towards obtaining a corrected version of the first identity on page $336$ of Ramanujan's Lost Notebook, that is, \eqref{ramin}. Note that Ramanujan's result has signs in the summands of the two infinite series exactly opposite to those in Corollary \ref{ramanujan-type}. For example, the infinite series on the right-hand side of \eqref{ramin} is the same as that on the right-hand side of \eqref{ramanujan-typecor} but with the $+$ sign in the summand of the latter replaced by $-$. The former would then be a special case of the series
\begin{align*}
\sum_{n=1}^{\infty}\sigma_s(n)\bigg\{&I_{\l}\left(2\pi\overline{\epsilon}\sqrt{n}(\sqrt{\a}-\sqrt{\b})\right)K_{\l}\left(2\pi\overline{\epsilon}\sqrt{n} (\sqrt{\a}+\sqrt{\b})\right)\nonumber\\
&-I_{\l}\left(2\pi\epsilon\sqrt{n}(\sqrt{\a}-\sqrt{\b})\right)K_{\l}\left(2\pi\epsilon\sqrt{n}(\sqrt{\a}+\sqrt{\b})\right)\bigg\}
\end{align*}
when $\l=-1/2$, and so it looks like an analysis similar to that carried out in this paper would lead to a corrected version of Ramanujan's formula that matches even in signs than the one we have given in \eqref{ramanujan-type}. Indeed, the natural kernel to work with while tackling this series would be $\sin\left(\frac{\pi s}{2}\right)J_{s}(t)-\cos\left(\frac{\pi s}{2}\right)L_{s}(t)$, where $L_{s}(t):=-\frac{2}{\pi}K_{s}(t)-Y_s(t)$. However, the problem is, we do not know of any summation formula like \eqref{guitra} which would involve this kernel.

It would be worthwhile to look for applications of \eqref{dixmolcorgeneqns0} (resp. \eqref{dixmolcorgeneqn}) towards the Dirichlet divisor problem (resp. the generalized Dirichlet divisor problem). This is because, the $r_2(n)$-analogue of \eqref{dixmolcorgeneqns0} was successfully used by Hardy to derive his famous omega result. Note that Hardy \cite{ddp} and subsequently many others, including Soundararajan \cite{soundararajanimrn}, base their study of the corresponding omega result for $d(n)$ on the conditionally convergent series
\begin{equation*}
\frac{\b^{1/4}}{\pi\sqrt{2}}\sum_{n=1}^{\infty}\frac{d(n)}{n^{3/4}}\cos\left(4\pi\sqrt{n\b}-\frac{\pi}{4}\right),
\end{equation*}
where as the series on the left-hand side of \eqref{dixmolcorgeneqns0} is different. To the best of our knowledge, no one has used an identity of the type \eqref{dixmolcorgeneqns0} in the study of Dirichlet divisor problem.

\begin{center}
\textbf{Acknowledgements}
\end{center}

\noindent
The authors sincerely thank Professor Anna Vishnyakova from V. N. Karazin Kharkiv National University for translating for them the last three pages of \cite{fockbursian}. They also sincerely thank Karrie Peterson, Head, MIT Libraries, for sending them a copy of \cite{schermann}, and Nico M.~Temme for informing them of the reference \cite{dunster}. The first author's research is supported by the SERB-DST grant RES/SERB/MA/P0213/1617/0021. He sincerely thanks SERB-DST for the support.

\end{document}